\documentclass[11pt]{amsart}
\usepackage{amssymb,amsmath,amsthm,amssymb}
\usepackage{fontawesome}
\usepackage{mathtools}
\usepackage{graphicx}
\usepackage{comment}
\usepackage{hyperref}
\usepackage{enumerate}
\usepackage{xcolor}
\usepackage{epstopdf}
\usepackage[normalem]{ulem}
\usepackage{epigraph}
\renewcommand{\epsilon}{\varepsilon}

\DeclareMathOperator{\dist}{dist}
\DeclareMathOperator{\Gaph}{graph}

\DeclareMathOperator{\spa}{span}

\DeclareMathOperator{\loc}{loc}

\def\R{\mathbb{R}}

\def\N{\mathbb{N}}

\def\a{\alpha}

\def\s{\sigma}

\def\wt{\widetilde}

\newtheorem{theorem}{Theorem}[section]
\newtheorem{lemma}[theorem]{Lemma}
\newtheorem{proposition}[theorem]{Proposition}

\newtheorem{remark}[theorem]{Remark}
\newtheorem{corollary}[theorem]{Corollary}

\newtheorem{claim}{Claim}[subsection]

\newtheoremstyle{TheoremNum}
        {\topsep}{\topsep}              
        {\itshape}                      
        {}                              
        {\bfseries}                     
        {.}                             
        { }                             
        {\thmname{#1}\thmnote{ \bfseries #3}}
    \theoremstyle{TheoremNum}

\def\ba #1\ea {\begin{align} #1\end{align}}
\def\bann #1\eann {\begin{align*} #1\end{align*}}
\def\ben #1\een {\begin{enumerate} #1\end{enumerate}}
\def\bi #1\ei {\begin{itemize}\renewcommand\labelitemi{--} #1\end{itemize}}

\title[Collapsing and noncollapsing in convex ancient MCF]{Collapsing and noncollapsing in convex ancient mean curvature flow}
\author{Theodora Bourni}
\author{Mat Langford}
\author{Stephen Lynch}

\address{Department of Mathematics, University of Tennessee Knoxville, Knoxville TN, 37996-1320}
\email{tbourni@utk.edu}
\email{mlangford@utk.edu}

\address{School of Mathematical and Physical Sciences, The University of Newcastle, Newcastle, NSW, Australia, 2308}
\email{mathew.langford@newcastle.edu.au}

\address{Eberhard Karls Universit\"at T\"ubingen, Fachbereich Mathematik, Auf der Morgenstelle 10, 72076 T\"{u}bingen, Germany}
\email{stephen.lynch@math.uni-tuebingen.de}

\date{\today}

\begin{document}

\begin{abstract}
We provide several characterisations of collapsing and noncollapsing in convex ancient mean curvature flow, establishing in particular that collapsing occurs if and only if the flow is asymptotic to at least one Grim hyperplane. As a consequence, we rule out collapsing singularity models in $(n-1)$-convex mean curvature flow (even when the initial datum is only immersed). 
Explicit counterexamples show that $(n-1)$-convexity is optimal. We are also able to rule out collapsing singularity models for suitably pinched solutions of higher codimension. 





\end{abstract}

\maketitle

\setcounter{tocdepth}{1}
\tableofcontents

\section{Introduction}

Ancient solutions to geometric flows (such as the mean curvature, Ricci or Yamabe flows) arise naturally as limits of sequences of rescalings in regions of large curvature, and therefore play a fundamental role in the analysis of singularity formation \cite{HamiltonSingularities}. Further interest in ancient solutions has resulted from their alluring geometric and rigidity properties. Both perspectives have motivated a series of recent classification results, a small cross-section of which can be found in \cite{ADS2,BamlerKleiner,BLTatomic,BLT3,BLT1,BrendleChoi2,BrendleDaskalopoulosNaffSesum,BrIvSc,BrLo,ChMa,MR3485129,MR3794888,DHScsf,DHSRicci,HuSi15,RisaSinestrari1,RisaSinestrari2,Wa11}. See \cite{MR3729051} for an early survey.

We are concerned here with ancient solutions to the \emph{mean curvature flow}. In this context, \emph{convex}\footnote{To be absolutely clear, a hypersurface is said to be (\emph{strictly}) \emph{convex} if it is the boundary of a (\emph{strictly}) \emph{convex body}, by which we mean a (\emph{strictly}) \emph{convex open set}. Given a convex hypersurface $M$, the shorthand $M=\partial \Omega$ will always mean that $\Omega$ is the convex body bounded by $M$. A \emph{convex mean curvature flow} is then a mean curvature flow whose timeslices are convex.} ancient solutions are of particular relevance due to the convexity estimate of Huisken and Sinestrari \cite{HuSi99b,HuSi99a}, which ensures that blow-up limits of compact, mean convex (not necessarily embedded) mean curvature flows are locally convex (though noncompact in general). In the embedded case this result was obtained independently by White \cite{Wh03}. 

The study of singularity formation in mean curvature flow has been advanced significantly following the establishment of \emph{noncollapsing} estimates (see \cite{An12,ShWa09,Wh00,Wh03}). These assert that the product of the inscribed radius with the mean curvature is bounded from below by a positive constant on the evolution of any compact \emph{embedded} hypersurface with positive mean curvature. On the other hand, noncollapsing can and does fail outside the setting of compact, embedded, codimension one, mean convex mean curvature flow. Indeed, 1. for solutions which are not embedded, not mean convex or not of codimension one, the noncollapsing property is ill-defined, 2. collapsing convex solutions such as the Grim Reaper can arise as blow-up limits in mean curvature flow of compact immersed \cite{Ang91} or compact embedded high codimension initial data \cite{MR3939057}, 3. there seems to be little hope of proving any kind of noncollapsing property for high codimension mean curvature flow since embeddedness is not even preserved, and 4. even for the most well-behaved fully nonlinear generalizations of mean curvature flow, noncollapsing is only known to hold if the speed is concave \cite{ALM13} (see also \cite{AL16}).
In each of these settings, the occurrence of collapsing singularity models is a major obstacle which must be overcome if, for example, the flow is to be extended through singularities via a surgery procedure. 

We shall provide here a new mechanism for ruling out collapsing singularity models (via curvature pinching), which applies even when solutions have self-intersections or high codimension. This sheds new light on singularity formation in immersed codimension one mean curvature flow, as well as high codimension mean curvature flow (see \cite{NaffPlanarity} and \cite{LynchNguyenConvexity}). Moreover, our tools are elementary, and thus appear to be sufficiently robust to apply to a much larger class of (fully nonlinear) hypersurface flows.

\begin{theorem}\label{thm:n-1 conv ancient}
If $\{\partial\Omega_t\}_{t\in(-\infty,0]}$ is a convex ancient mean curvature flow in $\R^{n+1}$ which is uniformly $(n-1)$-convex, then the inscribed radius of $\Omega_t$ is at least $H^{-1}$ at each point of $\partial\Omega_t$. 
\end{theorem}

Uniform $(n-1)$-convexity means there is some $\beta>0$ such that 
\[\lambda_1 + \dots + \lambda_{n-1} \geq \beta H\]
at each point of $\partial \Omega_t$, where $\lambda_1 \leq \dots \leq \lambda_n$ are the principal curvatures and $H$ is the mean curvature. This hypothesis cannot be weakened in Theorem \ref{thm:n-1 conv ancient}, since the product of $\mathbb{R}^{n-1}$ with a Grim Reaper (i.e. a  ``Grim hyperplane'') is collapsing. 


Theorem \ref{thm:n-1 conv ancient} is an immediate consequence of two results of independent interest, which we now describe. Let us call a convex ancient solution \emph{entire} 
if the region swept out by the enclosed bodies is $\mathbb{R}^{n+1}$. Of course, there are examples of both entire (the shrinking sphere, say) and non-entire (the pancake, say) ancient solutions. We shall establish that the class of entire convex ancient flows coincides exactly with those which are (optimally) interior noncollapsing. 



\begin{theorem}[Characterization of entire flows]\label{thm:noncollapsing equivalence}
Let $\{\partial\Omega_t\}_{t\in(-\infty,0]}$ be a convex ancient mean curvature flow in $\R^{n+1}$. The following are equivalent.
\begin{enumerate}
\item\label{cond:entire} $\cup_{t\le 0}\,\Omega_t=\R^{n+1}$.
\item\label{cond:noncollapsing} There exists $\alpha>0$ such that the inscribed radius of $\Omega_t$ is at least $\alpha H^{-1}$ at each point of $\partial\Omega_t$. 
\item\label{cond:optimal noncollapsing} The inscribed radius of $\Omega_t$ is at least $H^{-1}$ at each point of $\partial\Omega_t$. 
\end{enumerate}
\end{theorem}

To establish Theorem \ref{thm:noncollapsing equivalence} we prove that the space of entire convex ancient flows with $H$ normalised at the origin is compact in the smooth topology. From this the inscribed radius estimate is immediate. Brendle and Naff arrived at Theorem \ref{thm:noncollapsing equivalence} independently by proving a local noncollapsing estimate \cite{Brendle-Naff}, and there is yet another approach that uses Huisken's monotonicity formula due to Choi--Haslhofer--Hershkovits \cite{CHH21} (see the proof of Theorem 4.2).

The second ingredient needed to prove Theorem \ref{thm:n-1 conv ancient} is the following characterization of convex ancient mean curvature flows which fail to be entire. In particular, we show that every such flow contains arbitrarily large (on the scale of the curvature) spacetime regions modelled on a Grim hyperplane. 

\begin{theorem}[Characterization of non-entire flows]\label{thm:Reapers}
Let $\{\partial\Omega_t\}_{t\in(-\infty,0]}$ be a convex ancient mean curvature flow in $\R^{n+1}$ with $H>0$. The following are equivalent.
\begin{enumerate}
\item\label{cond:not entire} $\cup_{t\in(-\infty,0]}\,\Omega_t\subsetneq \R^{n+1}$. 
\item\label{cond:slab} $\cup_{t\in(-\infty,0]}\,\Omega_t$ is a slab region. 
\item\label{cond:collapsing} $\{\partial\Omega_t\}_{t\in(-\infty,0]}$ is collapsing (in the sense of inscribed radii). 
\item\label{cond:some blowdown} The family $\{r \partial\Omega_{r^{-2}t}\}_{t\in(-\infty,0]}$ converges locally smoothly to a hyperplane of multiplicity two as $r \to 0$.
\item\label{cond:Grim} $\{\partial\Omega_t\}_{t\in(-\infty,0]}$ admits a sequence of rescaled spacetime translations which converges locally smoothly to a Grim hyperplane.
\end{enumerate}
\end{theorem}
\begin{remark}
In part \eqref{cond:some blowdown} of Theorem \ref{thm:Reapers}, the convergence is in the intrinsic sense. 
In \eqref{cond:Grim}, the convergence is in the extrinsic sense. 
\end{remark}

\subsection{Singularities} Theorem \ref{thm:n-1 conv ancient} yields the surprising corollary that singularity models in $(n-1)$-convex mean curvature flow are noncollapsed, even if the original solution has self-intersections. To put this into context, note that $(n-1)$-convexity is a strictly weaker condition than positive scalar curvature when $n\ge 3$. 

\begin{corollary}[Noncollapsing in $(n-1)$-convex mean curvature flow]\label{cor:n-1 convex}
Let $\{M_t\}_{t \in (-\infty, 0]}$ be a proper, nonflat ancient mean curvature flow in $\R^{n+1}$ 
which arises as a smooth blow-up limit of a compact mean curvature flow with uniformly $(n-1)$-convex (not necessarily embedded) initial datum. Each timeslice $M_t$ is the boundary of a convex domain $\Omega_t$, and the inscribed radius of $\Omega_t$ is at least $H^{-1}$ at every point in $M_t$. 
\end{corollary}
 
For \emph{embedded} flows, Corollary \ref{cor:n-1 convex} is, of course, a consequence of the known noncollapsing results \cite{An12,ShWa09,Wh03}, but it goes far beyond the best currently known results for the general case (which are based on the Huisken--Sinestrari gradient estimate \cite[Theorem 6.1]{HuSi09}). In fact, the result is sharp in that there are examples of compact, mean convex initial data that form singularities modelled on a Grim hyperplane, which is collapsing and weakly $(n-1)$-convex --- consider a very long torus over a one-dimensional cusp \cite{Ang91}. Moreover, the Grim hyperplane singularities in these examples cannot be perturbed away.

The following is the analogue of Corollary \ref{cor:n-1 convex} for solutions of higher codimension which satisfy the pinching condition introduced by Andrews--Baker \cite{AnBa10}.

\begin{corollary}[Noncollapsing in pinched high codimension mean curvature flow]\label{cor:pinched high codim}
Let $\{M_t\}_{t \in (-\infty, 0]}$ be a proper, nonflat, $n$-dimensional ancient mean curvature flow in $\R^{n+k}$ 
which arises as a smooth blow-up limit of a compact mean curvature flow in $\R^{n+k}$ with initial datum satisfying $|A|^2 \leq \{\tfrac{4}{3n}, \tfrac{3(n+1)}{2n(n+2)}\}|H|^2$. Up to a fixed rotation, each timeslice $M_t$ is the boundary of a convex domain $\Omega_t$ in $\R^{n+1}$, and the inscribed radius of $\Omega_t$ is at least $H^{-1}$ at every point in $M_t$. 
\end{corollary}


Another implication of Theorem \ref{thm:n-1 conv ancient} is the relaxation of noncollapsing hypotheses in various classification results for ancient solutions. We illustrate this by removing the noncollapsing hypothesis in the work of Angenent--Daskalopoulos--\v{S}e\v{s}um \cite{ADS2} and Brendle--Choi \cite{BrendleChoi2}, resulting in a classification of convex uniformly two-convex ancient mean curvature flows of dimension at least three. 

\begin{corollary}\label{cor:classification}
The shrinking sphere, the admissible shrinking cylinder, the admissible ancient ovaloid, and the bowl soliton are the only convex ancient mean curvature flows in $\R^{n+1}$, $n\ge 3$, which are uniformly two-convex.
\end{corollary}

\subsection{Remarks on related work}
In \cite{Naff19}, Naff proved Corollary \ref{cor:pinched high codim} for solutions satisfying $|A|^2 \leq \{\tfrac{1}{n-2}, \tfrac{3(n+1)}{2n(n+2)}\}|H|^2$, which is a quadratic analogue of two-convexity. He also used the gradient estimate of Huisken and Sinestrari \cite{HuSi09} to prove Corollary \ref{cor:n-1 convex} for two-convex solutions of dimension at least three. In the  Huisken--Sinestrari gradient estimate, two-convexity can be weakened to $k$-convexity so long as $k<\frac{2n+1}{3}$, with only minor modifications to the proof (see, for example, \cite[Theorem 9.24]{EGF}). Huisken and Sinestrari have also proved an analogous gradient estimate for $k$-convex ancient solutions, $k<\frac{2n+1}{3}$, assuming in addition that $\min_{M_t}H^2$ is not integrable \cite{HuSi15}. It does not appear to be possible to weaken the strong intermediate convexity hypothesis (observe that $\tfrac{2n+1}{3}< n-1$ for all $n \geq 5$) in these arguments since they depend delicately on the constant in the Kato inequality $|\nabla A|^2 \leq \tfrac{3}{n+2} |\nabla H|^2$.

By \cite{LynchUniqueness}, all convex ancient flows with type-I curvature growth are cylinders, and thus are noncollapsing, so Theorem \ref{thm:n-1 conv ancient} is vacuous in this case. Of course, the vast majority of convex ancient flows are not of type-I.

Choi, Haslhofer, Hershkovits and White \cite{CHHW} have proved a statement which is related to Corollary \ref{cor:classification}. Rather than assuming convexity and uniform two-convexity, their hypothesis is that the blow-down is a shrinking sphere or cylinder: $\{\R^m\times S^{n-m}_{\sqrt{-2(n-m)t}}\}_{t\in(-\infty,0)}$, $m \in \{0,1\}$.

\subsection{Outline}
After collecting auxiliary results in Sections \ref{sec:preliminaries} and \ref{sec:concavity}, we prove Theorem \ref{thm:noncollapsing equivalence} in Section \ref{sec:compactness} and Theorem \ref{thm:Reapers} in Section \ref{sec:Reapers}. The proofs of Corollaries \ref{cor:n-1 convex} and \ref{cor:classification} are contained in Section \ref{sec:singularities}. 

\subsection*{Acknowledgements} M.~Langford acknowledges support from the Australian Research Council through the DECRA fellowship scheme (grant DE200101834). T.~Bourni acknowledges support from the Simons foundation (grant 707699) and the National Science Foundation (grant DMS-2105026). S.~Lynch is grateful to K.~Naff for a number of discussions which have benefited this work.

\section{Preliminaries}\label{sec:preliminaries}

The following Ecker--Huisken-type interior estimate for convex solutions to mean curvature flow will be invaluable in our analysis.

\begin{proposition}[Interior curvature estimates]
\label{prop:interior_est}
There is a constant $C=C(n)$ with the following property. Let $\{\partial\Omega_t\}_{t \in [-T, 0]}$ be a convex solution of mean curvature flow in $\R^{n+1}$. If $B_r(0)$ is contained in $\Omega_0$, then
\begin{equation}\label{eq:int_1}
\sup_{B_R(0) \times [- T , 0]}   H \leq \max\Big\{ \max_{B_{2R}(0)} H(\cdot, -T), Cr^{-4} R^3  \Big\},
\end{equation}
and 
\begin{equation}\label{eq:int_2}
\sup_{B_R(0) \times [- T/2 , 0]}   H \leq C \Big(1 + r^{-1} \sqrt{T}\Big)r^{-4}  R^3.
\end{equation}
\end{proposition}  
\begin{proof}
The argument is a direct adaptation of \cite[Proposition 5.1]{BrHu17}. 

By convexity, we know that the function $v\doteqdot |x|^{-2}(x \cdot \nu)^2$ satisfies $v \geq \theta$ at points $x \in B_{2R}(0)$, where $\theta \doteqdot (1+4r^{-2} R^{2})^{-1}$. We have the evolution equation 
\begin{align*}
(\partial_t - \Delta) v & =  2(|A|^2- 2(x\cdot \nu)^{-1}H) v - \frac{|\nabla v|^2}{2 v} + |x|^{-2} \langle \nabla v, \nabla |x|^2\rangle\\
&+2 |x|^{-2} v \bigg(n  - \frac{|\nabla |x|^2|^2}{4|x|^2} \bigg)\\
& \geq 2(|A|^2- 2|x|^{-1}H) v - \frac{|\nabla v|^2}{2 v} + |x|^{-2} \langle \nabla v, \nabla |x|^2\rangle.
\end{align*}
Let $\varphi(v) := \sqrt{v - \theta/2}$ and define $\psi:= \varphi(v)^{-1} H$. We then compute 
\begin{align*}
(\partial_t - \Delta) \psi & = |A|^2 \psi - \varphi(v)^{-1} \dot \varphi(v) \psi (\partial_t - \Delta)v + \varphi(v)^{-1} \ddot \varphi(v) \psi |\nabla v|^2 \\
&+2 \varphi(v)^{-1} \dot \varphi(v) \langle \nabla \psi, \nabla v \rangle.
\end{align*}
Inserting the lower bound for $(\partial_t - \Delta) v$ derived above, we obtain 
\begin{align*}
(\partial_t - \Delta) \psi &\leq |A|^2 \psi \Big(1  - 2\varphi(v)^{-1} \dot \varphi(v) v\Big) + 4 |x|^{-1} \varphi(v)^{-1} \dot \varphi(v) v H \psi\\
&+ \varphi(v)^{-1} \dot \varphi(v) \psi \Big(\frac{|\nabla v|^2}{2 v} - |x|^{-2} \langle \nabla v, \nabla |x|^2\rangle\Big)\\
&+ \varphi(v)^{-1} \ddot \varphi(v) \psi |\nabla v|^2 +2 \varphi(v)^{-1} \dot \varphi(v) \langle \nabla \psi, \nabla v \rangle.
\end{align*}
Since $\varphi(v)^{-1} \dot \varphi(v) v = \frac{v/2}{v - \theta/2}$ we have
\begin{align*}
|A|^2 \psi \Big(1  - 2\varphi(v)^{-1} \dot \varphi(v) v\Big)& + 4 |x|^{-1} \varphi(v)^{-1} \dot \varphi(v) v H \psi\\
& \leq - \frac{\theta}{4n} \psi^3 +8n \theta^{-2} r^{-2} \psi,
\end{align*}
and by Young's inequality
\begin{align*}
\varphi(v)^{-1} \dot \varphi(v) \psi \Big(\frac{|\nabla v|^2}{2 v} - |x|^{-2} &\langle \nabla v, \nabla |x|^2\rangle\Big) \\
&\leq  \frac{1}{4} (1+\varepsilon) \psi \frac{|\nabla v|^2}{v(v - \theta /2)} + 2 \theta^{-1} \varepsilon^{-1} r^{-2} \psi
\end{align*}
so using
\begin{align*}
\varphi(v)^{-1} \ddot \varphi(v) \psi |\nabla v|^2 := -\frac{1}{4} \psi \frac{|\nabla v|^2}{(v - \theta/2)^2}
\end{align*} 
we arrive at
\begin{align*}
(\partial_t - \Delta) \psi &\leq - \frac{\theta}{4n} \psi^3 +8n \theta^{-2} r^{-2} \psi + 2 \theta^{-1} \varepsilon^{-1} r^{-2} \psi\\
&+  \frac{1}{4}  \Big( \varepsilon - \theta/2\Big) \psi \frac{|\nabla v|^2}{v(v - \theta/2)^2}+2 \varphi(v)^{-1} \dot \varphi(v) \langle \nabla \psi, \nabla v \rangle.
\end{align*}

We define $\eta = 4R^2 - |x|^2$, and recall that $(\partial_t - \Delta) \eta^2 = 2n$. Writing $h := \eta \psi$, we compute 
\begin{align*}
(\partial_t - \Delta) h & =   - \frac{\theta}{4n} \eta^{-2} h^3 +8n \theta^{-2} r^{-2} h + 2 \theta^{-1} \varepsilon^{-1} r^{-2} h\\
& + 2n \eta^{-1} h +  \frac{1}{4}  \Big( \varepsilon - \theta/2\Big) h  \frac{|\nabla v|^2}{v(v - \theta/2)^2}\\
&+2 \varphi(v)^{-1} \dot \varphi(v) \langle \eta \nabla \psi, \nabla v \rangle - 2 \langle \nabla \psi, \nabla \eta\rangle.
\end{align*}
We estimate the final two gradient terms by 
\begin{align*}
2 \varphi(v)^{-1} \dot \varphi(v) &\langle \eta \nabla \psi, \nabla v \rangle - 2 \langle \nabla \psi, \nabla \eta\rangle \\
& \leq \frac{1}{4} \varepsilon h \frac{|\nabla v|^2}{v(v-\theta/2)^2} + 4(2+\varepsilon^{-1}) R^2 \theta^{-1}\eta^{-2} h \\
& +  (v-\theta/2)^{-1} \langle \nabla h, \nabla v \rangle - 2 \eta^{-1} \langle \nabla h , \nabla \eta\rangle,
\end{align*}
and substitute back in to arrive at 
\begin{align*}
(\partial_t - \Delta) h &=   - \frac{\theta}{4n} \eta^{-2} h^3 +8n \theta^{-2} r^{-2} h + 4(2+\varepsilon^{-1}) R^2 \theta^{-1}  \eta^{-2} h \\
& +2n \eta^{-1} h+  \frac{1}{4}  \Big( 2\varepsilon - \theta/2\Big) h  \frac{|\nabla v|^2}{v(v - \theta/2)^2}\\
&+  (v-\theta/2)^{-1} \langle \nabla h, \nabla v \rangle - 2 \eta^{-1} \langle \nabla h , \nabla \eta\rangle.
\end{align*}
Setting $\varepsilon = \theta/4$, we find there is a $C = C(n)$ such that 
\begin{align*}
(\partial_t - \Delta) h &=   - \frac{\theta}{4n} \eta^{-2} h^3  + C \theta^{-2} r^{-2} h + C R^2 \theta^{-2}  \eta^{-2} h +C\eta^{-1} h  \\
&+  (v-\theta/2)^{-1} \langle \nabla h, \nabla v \rangle - 2 \eta^{-1} \langle \nabla h , \nabla \eta\rangle.
\end{align*}

Suppose $h$ attains its maximum at an interior point $(\bar x, \bar t)$. We then have 
\[\frac{\theta}{4n} \eta^{-2} h^3 \leq  C \theta^{-2} r^{-2} h + C R^2 \theta^{-2}  \eta^{-2} h + C\eta^{-1} h \]
at $(\bar x, \bar t)$, and hence $h(\bar x, \bar t)^2 \leq C \theta^{-3} R^4 r^{-2}$. Inserting the definition of $\theta$, this gives $h(\bar x, \bar t)^2 \leq  C R^{10} r^{-8}$. We thus deduce \eqref{eq:int_1}.

Suppose now that $(t+T)^\frac{1}{2} h$ attains its maximum at $(\bar x, \bar t)$. We then have 
\[\frac{\theta}{4} \eta^{-2} h^3 \leq \frac{1}{2}(\bar t+T)^{-1}h +  C\eta^{-1} h  + C \theta^{-2} r^{-2} h + C R^2 \theta^{-2}  \eta^{-2} h \]
at $(\bar x, \bar t)$. We rearrange to obtain $(\bar t+T)^\frac{1}{2} h(\bar x, \bar t) \leq C (1 + r^{-1}T^\frac{1}{2}) r^{-4} R^{5} $, which immediately implies \eqref{eq:int_2}. 
\end{proof}

The interior estimate implies the following rudimentary compactness property for the space of convex ancient solutions. This generalises the global convergence theorem of Haslhofer--Kleiner \cite{HK1} in the special case of convex flows. 

Recall that the dimension of a convex set $\Omega\subset \R^{n+1}$ is defined to be the minimum of $\dim(\Sigma)$ taken across all affine subspaces $\Sigma\subset \R^{n+1}$ containing $\Omega$. 

\begin{proposition}
\label{prop:compactness}
Let $\{M_t^i = \partial \Omega_t^i\}_{t \in (-\infty,0]}$ be a sequence of convex ancient mean curvature flows, and $x_i \in M_{t_i}^i$ a sequence of points such that $(x_i, t_i) \to (0,0)$ and 
\[\liminf_{i \to \infty} H^i(x_i,t_i) >0.\]
The following are equivalent.
\begin{enumerate}
\item \label{cond:compact_smooth} A subsequence of the flows $\{M_t^i\}_{t \in (-\infty,0]}$ converges in $C^\infty_{\loc}(\mathbb{R}^{n+1} \times (-\infty,0])$. 
\item \label{cond:compact_H} There are constants $\rho>0$ and $C<\infty$ such that, after passing to a subsequence,
\[\sup_{B_{\rho}(0) \times [-\rho^2,0]} H^i \leq C.\]
\item \label{cond:compact_Hdorff} The sequence $\Omega_0^i$ subconverges in the Hausdorff topology to a convex set of dimension $n+1$. 
\item \label{cond:compact_balls} After passing to a subsequence,  there is an open ball in $\cap_{i \in \mathbb{N}}\,\Omega_0^i$. 
\end{enumerate}
\end{proposition}
\begin{proof}
Clearly \eqref{cond:compact_smooth} implies \eqref{cond:compact_H}. If \eqref{cond:compact_H} holds, then, since $|A^i| \leq H^i$, the Ecker--Huisken interior estimates for derivatives of  curvature \cite{EckerHuisken91} imply there is a radius $r \in (0,\rho]$ such that the sequence of hypersurfaces $M_{t_i}^i \cap B_{r}(x_i)$ subconverges in $C^\infty$ to a limit $M'$ (\emph{a priori} with the possibility of multiplicity, though this is ruled out once we establish \eqref{cond:compact_Hdorff}). The mean curvature of $M'$ is positive at the origin, so $M'$ does not lie in any hyperplane in $\mathbb{R}^{n+1}$. Passing to a further subsequence, we may assume $\overline \Omega{}_0^i$ converges in the Hausdorff topology to a closed convex set $K$. Since $M' \subset K$ we conclude $K$ does not lie in a hyperplane, or equivalently, the dimension of $K$ is $n+1$. That is, \eqref{cond:compact_Hdorff} holds, and by convexity there is an open ball contained in $\cap_{i \in \mathbb{N}} \, \Omega_0^i$,  hence \eqref{cond:compact_balls} holds.  Finally, given \eqref{cond:compact_balls}, Proposition \ref{prop:interior_est} implies $H^i$ is bounded independently of $i$ in compact subsets of $\mathbb{R}^{n+1} \times (-\infty,0]$, so using the Ecker--Huisken estimates \cite{EckerHuisken91} we obtain \eqref{cond:compact_smooth}.
\end{proof}

The next lemma shows that lower dimensional slices of convex mean curvature flows are subsolutions to mean curvature flow (that is, their inward normal speed is greater than or equal to their mean curvature).

Given a unit vector $e \in \mathbb{R}^{n+1}$ we write $\mathbb{R}e \doteqdot \{s e : s \in \mathbb{R}\}$ and, similarly, $\mathbb{R}_{\pm} e \doteqdot \{s e : \pm s \geq 0\}$.

\begin{lemma}
\label{lem:slice_subsols}
Let $\{M_t=\partial \Omega_t\}_{t\in(-T,0]}$ be a strictly convex mean curvature flow. Let $\Sigma \subset \mathbb{R}^{n+1}$ be an affine subspace of dimension $k+1$, and suppose that $\Omega_0 \cap \Sigma$ is nonempty. The slice $\tilde M_t= M_t \cap \Sigma$ is a smooth, convex, $k$-dimensional hypersurface in $\Sigma$ for each $t \in (-T, 0]$, and the family $\{\tilde M_t\}_{t\in(-T,0]}$ is a subsolution to mean curvature flow in $\Sigma$. 
\end{lemma}
\begin{proof}
The Lemma is proved in the course of the proof of \cite[Lemma 2.9]{Wa11}. We reproduce the argument here for the convenience of the reader.

We first observe that $\tilde M_t$ is smooth hypersurface in $\Sigma$ for each $t\in(-T,0]$. This is a consequence of the inverse function theorem, provided there is no point $x \in \tilde M_t$ at which $\nu(x,t)$ (the outward unit normal to $M_t$) is normal to $\Sigma$. To see that this is indeed the case, observe that if $x \in \tilde M_t$ is such that $\nu(x,t) \in \Sigma^\perp$, then $\Sigma$ is tangent to $M_t$ at $x$, and hence $\Sigma \cap M_t$ is a point by strict convexity. This is impossible, given that $\Omega_0$ is open, $\Omega_0\cap\Sigma$ is nonempty and $\Omega_0\subset \Omega_t$ for $t\le 0$. 
Therefore, $\tilde M_t$ is smooth.

Let $\tilde \nu$ denote the outward unit normal $\tilde M_t$ in $\Sigma$. Fix a point $x_0 \in \tilde M_{t_0}$. To prove the lemma, we need to find a parametrization $\tilde X$ of $\tilde M_t$ in a neighbhourhood of $(x_0, t_0)$ such that 
\[\partial_t \tilde X (p_0, t_0) \cdot \tilde \nu (x_0, t_0) \leq - \tilde H(x_0,t_0),\]
where $\tilde X(p_0, t_0)=x_0$ and $\tilde H$ is the mean curvature of $\tilde M_t$.

Choose Euclidean coordinates for $\mathbb{R}^{n+1}$ such that $x_0=0$, $\Sigma$ coincides with $\spa\{e_1, \dots, e_k, e_{n+1}\}$, and $\tilde \nu(0, t_0) = e_{n+1}$. In particular, $e_i$ is tangent to $\tilde M_{t_0}$ at $0$ for $1 \leq i \leq k$. We place no further restrictions on $\{e_{k+1}, \dots, e_{n}\}$, they simply generate $\Sigma^\perp$. 

Since $e_{n+1}$ is normal to $\tilde M_{t_0}$, the set $\Omega_{t_0} \cap \mathbb{R} e_{n+1}$ is an open interval. On the other hand, if $\ell$ is a line tangent to $M_t$, strict convexity implies $\ell \cap \overline{\Omega}_t$ is a point. Combining these observations we conclude $\nu(0,t_0) \cdot e_{n+1} >0$. Therefore, by the inverse function theorem, there is some $r>0$ and a smooth function $u$ defined on $Q_r\doteqdot B^n_r(0) \times [t_0 - r^2 , t_0]$ such that  
\[X(p,t) \doteqdot p + u(p,t)e_{n+1},\]
defines a local parametrization for $M_t$, and $X(0,t_0) = 0$. It follows that for $\tilde Q_r \doteqdot (B^n_r(0) \cap \Sigma) \times [t_0 - r^2, t_0]$ and $\tilde u \doteqdot  u|_{\tilde Q_r}$, the map 
\[\tilde X(p,t) = p + \tilde u(p,t) e_{n+1}\]
defines a local parametrization of $\tilde M_t$. 

Notice we have 
\[\partial_t \tilde X(0,t_0) \cdot \tilde \nu(0,t_0) = \partial_t \tilde X(0,t_0) \cdot e_{n+1} = \partial_t \tilde u(0,t_0) = \partial_t u(0,t_0).\]
On the other hand, since the $M_t$ move by mean curvature flow,
\begin{align*}
\frac{\partial_t u}{\sqrt{1+|D u|^2}} = \partial_t X \cdot \nu = - H = \Big( \delta_{ij} - \frac{D_i u D_j u}{1+|Du|^2}\Big) \frac{D_i D_j u}{\sqrt{1+|Du|^2}},
\end{align*}
so since $D_i u (0,t_0) = D_i\tilde u(0,t_0) = 0$ for $ 1\leq i \leq k$, at the point $(0,t_0)$ we have
\begin{align*}
\partial_t u &=  \Big( \delta_{ij} - \frac{D_i u D_j u}{1+|Du|^2}\Big) D_i D_j u\\
&=  \sum_{i =1}^k D_i D_i u + \sum_{i,j = k+1}^n \Big( \delta_{ij} - \frac{D_i u D_j u}{1+|Du|^2}\Big) D_i D_j u\\
&\leq  \sum_{i =1}^k D_i D_i u \\
&= -\tilde H.
\end{align*}
Here we have made crucial use of the fact that $M_t$ is convex, hence $u$ is concave in its spatial variables. In summary we have shown
\[\partial_t \tilde X(0,t_0) \cdot \tilde \nu(0,t_0) = \partial_t \tilde u (0,t_0) = \partial_t u(0,t_0) \leq - \tilde H(0, t_0),\]
as required.
\end{proof}

Finally, we recall the following well-known ``point-selection'' trick.

\begin{lemma}\label{lem:pp}
Let $\{M_t\}_{t \in [\alpha, \omega]}$ be a complete solution of mean curvature flow. If $t\in(\alpha,\omega]$ and $p\in M_{t}$ satisfy $H(p,t) \geq r^{-1}$ and $\alpha\leq t-4r^2$, then there is a point $(\bar p,\bar t) \in  B_{2 r}(p) \times [t- 2r^2, t]$ satisfying $H(\bar p, \bar t) \geq H(p,t)$ and 
\begin{equation}\label{eq:point picked}
\sup_{B_{\bar r}(\bar p) \times [\bar t - \bar r^2,  \bar t]} H \leq 2 / \bar r\,,
\end{equation}
where $\bar r \doteqdot 1 / H(\bar p,  \bar t)$. 
\end{lemma} 
\begin{proof}
We follow \cite{EckerBook}. If $(p, t)$ satisfies \eqref{eq:point picked}, then we may take $(\bar p , \bar t)=(p,t)$. Else, we can find $(p_1 ,t_1)\in B_{r}(0) \times [t-r^2, t]$ such that $H(p_1,t_1)>2H(p,t)$. If $(p_1,t_1)$ satisfies
\[
\sup_{B_{\frac{1}{H(p_1,t_1)}}(p_1) \times [t_1-\frac{1}{H(p_1,t_1)^{2}}, t_1]} H \leq 2H(p_1,t_1)\,,
\]
then we take $(\bar p , \bar t)=(p_1,t_1)$. Since $\sum_{k=0}^\infty 2^{-k}=2$ and $H$ is finite in compact subsets of spacetime, continuing in this way we find, after some finite number of steps $k$, a point $(\bar p,\bar t)\doteqdot (p_k,t_k)$ with the desired properties.
\end{proof}

\section{Concavity of the arrival time}\label{sec:concavity}

Given a convex mean curvature flow $\{\partial \Omega_t\}_{t \in I}$, we define its arrival time $u:\cup_{t\in I} \, \Omega_t\to I$ by
\[u(x) := \sup\{t \in I: x \in \Omega_t\}.\]

Evans and Spruck \cite{ES91} (see also \cite{Trudinger90} and \cite{BLconcavity}) showed that every compact convex mean curvature flow has root-concave arrival time (i.e. the function $\sqrt{2(u-u_0)}$ is concave, where $u_0:=\inf I$). Since root-concavity of the arrival time is locally equivalent to Hamilton's differential Harnack inequality, the latter implies root-concavity of the arrival time for noncompact flows under the additional assumption of bounded curvature on compact time intervals \cite{BLconcavity}. Recent work of Daskalopoulos--Saez \cite{DaskSaez} may be used to remove this hypothesis. In particular, we find that every convex \emph{ancient} mean curvature flow has \emph{concave} arrival time (equivalently, satisfies the differential Harnack inequality).



\begin{remark}\label{rem:Wang is a Wang1}
In an earlier preprint, written before \cite{DaskSaez} became available, we were able to obtain Theorems \ref{thm:noncollapsing equivalence} and \ref{thm:Reapers} without appealing to concavity of the arrival time. However, this property can be used to simplify our arguments in several places; for the sake of clarity, we have decided to incorporate it into the present version.
\end{remark}

\begin{proposition}\label{arrival concave}
Every convex mean curvature flow $\{\partial\Omega_t\}_{t\in I}$ has root-concave arrival time. In particular, every convex ancient mean curvature flow has concave arrival time.
\end{proposition}


The key to proving Proposition \ref{arrival concave} is to show that every noncompact convex flow admits a smooth approximation by compact flows. Wang argued that this is the case in \cite[Proposition 4.1]{Wa11}, but his proof implicitly assumes that solutions of mean curvature flow out of convex initial data are unique. Until recently, this was only known under restrictive additional hypotheses, such as an upper bound for the second fundamental form. The work of Daskalopoulos--Saez \cite{DaskSaez} improves the situation substantially, establishing uniqueness for graphical convex initial data under extremely mild conditions, and in particular in the absence of any upper bound for curvature. Using their results, Daskalopoulos--Saez proved that convex flows which can be written as entire graphs admit a smooth approximation by compact flows, and thus proved that the differential Harnack inequality holds in this setting \cite[Section 5]{DaskSaez}. We will see that a straightforward modification of their argument (using Proposition~\ref{prop:interior_est} rather than the Ecker--Huisken curvature estimate for graphs over hyperplanes) yields the general case. We require the following uniqueness statement.

\begin{lemma}\label{lem:convex uniqueness}
Let $\{\partial \Omega_t\}_{t \in [0,T]}$ be a smooth, noncompact, convex, locally uniformly convex mean curvature flow. If $t_0>0$ and $\{\partial \tilde \Omega_t\}_{t \in [t_0, T]}$ is a smooth convex mean curvature flow such that $\tilde \Omega_{t_0} = \Omega_{t_0}$, then $\partial \tilde \Omega_t = \partial \Omega_t$ for $t \in [t_0,T]$. 
\end{lemma}

In the following argument we write $\mathcal T_\infty \Omega$ for the asymptotic cone of an open convex set $\Omega$, which is defined as the Hausdorff limit of $k^{-1} \overline \Omega$ as $k \to \infty$, or alternatively the largest closed convex cone with vertex at the origin which is contained in a translate of $\Omega$.

\begin{proof}[Proof of Lemma \ref{lem:convex uniqueness}]
The desired uniqueness follows from results in \cite{DaskSaez}, provided the two solutions can be written as graphs over the same (evolving) domain, and in such a way that their height functions are each bounded from below. 

A straightforward argument using the avoidance principle shows that $\mathcal T_\infty \Omega_t$ is independent of $t$. Indeed, if we shift $\Omega_T$ so that $B_r(0) \subset \Omega_T$ for some $r >0$, then $0 \in \Omega_0$, and hence $\mathcal T_\infty \Omega_0 \subset \Omega_0$. Define 
\[\bar t := \sup\{t \in [0,T] : \mathcal T_\infty \Omega_0 \subset \overline \Omega_t\}.\]
We then have $\mathcal T_\infty \Omega_0 \subset \overline \Omega_{\bar t}$, but since $B_r(0) \subset \Omega_{\bar t}$, convexity lets us conclude that $B_r(v) \subset \Omega_0$ for all $v \in \mathcal T_\infty \Omega_0$. Therefore, if $\bar t < T$ then the avoidance principle implies that $\mathcal T_\infty \Omega_0 \in \overline \Omega_t$ for all $t \in [0, \min\{\bar t + r^2/2n, T\}]$, which contradicts the definition of $\bar t$, so we have $\bar t = T$. We thus conclude that $\mathcal T_\infty \Omega_0 \subset \mathcal T_\infty \Omega_t$ for each $t \in [0,T]$. On the other hand $\Omega_t \subset \Omega_0$ implies $\mathcal T_\infty \Omega_t \subset \mathcal T_\infty \Omega_0$ for each $t \in [0,T]$, so $\mathcal T_\infty \Omega_t = \mathcal T_\infty \Omega_0$ for each $t \in [0,T]$. Therefore, let us write $K$ for $\mathcal T_\infty \Omega_t$. 

Since $\Omega_t$ is locally uniformly convex, $K$ does not contain any lines, so up to a rotation we may assume $\mathbb{R}_+ e_{n+1} \subset K$, $K\setminus \{0\} \subset \{x_{n+1} >0\}$, and that $K \cap \{x_{n+1} = h\}$ is compact for each $h \geq 0$. In particular, these properties imply that $\partial \Omega_t$ is a graph over an open subset $D_t \subset \{x_{n+1} = 0\}$ for each $t \in [0,T]$, and since $\overline \Omega_0 \cap \{x_{n+1} \leq 0\}$ is compact, the height function is bounded from below independently of $t$. Similarly, since $\mathcal T_\infty \tilde \Omega_{t}$ is constant and equal to $\mathcal T_\infty \Omega_{t_0} = K$, $\partial \tilde \Omega_t$ is a graph over an open subset $\tilde D_t \subset \{x_{n+1} = 0\}$ with height bounded from below independently of $t \in [t_0, T]$. We have $D_{t_0} = \tilde D_{t_0}$.
 
In case $D_0 = \mathbb{R}^n$, the avoidance principle implies $D_t = \mathbb{R}^n$ for all $t \in [0,T]$, in which case the claim follows immediately from \cite[Theorem 1.3]{DaskSaez}. Suppose then that $D_0 \not = \mathbb{R}^n$.

We now prove the claim by induction on the dimension $n$, so suppose first that $n=1$. Observe that $\mathbb{R} \times \partial \Omega_t$ is a convex mean curvature flow in $\mathbb{R}^3$, which is a graph over $\mathbb{R} \times D_t$. We label axes so that $e_3$ is the graphical direction, the extra $\mathbb{R}$-factor corresponds to the $e_1$-axis, and $D_t$ lies on the $e_2$-axis. Let $B$ be an open ball in $\mathbb{R}\times\Omega_T$. We then have that 
\[B \subset \mathbb{R}\times  \Omega_t - k e_3 \]
for all $t \in [0, T]$ and $k \in \mathbb{N}$, and since $D_{t_0} \not = \mathbb{R}$, the minimum distance from $M_{t}^k\doteqdot \mathbb{R}\times  \partial\Omega_t - k e_3$ to the origin is bounded from above independently of $k$. Therefore, by the interior estimate \eqref{eq:int_2} of Proposition \ref{prop:interior_est} and the Ecker--Huisken interior derivative estimates \cite{EckerHuisken91}, the sequence of flows $M_t^k$ subconverges in $C^\infty_{\loc}(\mathbb{R}^3 \times (0,T])$ as $k \to \infty$. Moreover, the limit is $\mathbb{R} \times \partial D_t \times \mathbb{R}$. In particular, $\{\mathbb{R}\times\partial D_t\}_{t\in [{t_0}, T]}$ is a smooth mean curvature flow in $\mathbb{R}^2$. 

We now give a similar argument showing that $\{\mathbb{R} \times \partial \tilde D_t\}_{t\in [{t_0}, T]}$ is a smooth mean curvature flow in $\mathbb{R}^2$; to prove this for the closed interval $[t_0,T]$, we will exploit the fact that $t_0$ is an interior time for $\Omega_t$. Estimate \eqref{eq:int_1} of Proposition \ref{prop:interior_est} implies that the hypersurfaces $\tilde M_t^k \doteqdot \mathbb{R} \times \partial \tilde \Omega_t - k e_3$ have mean curvature bounded from above independently of $k$ on compact subsets of $\mathbb{R}^{3} \times [t_0, T]$, since there is an open ball $\tilde B$ contained in $\mathbb{R} \times \tilde \Omega_{T} - k e_3$ for all $k$, and the mean curvature of $\tilde M_{t_0}^k = M_{t_0}^k$ is bounded independently of $k$ on compact subsets of $\mathbb{R}^{3}$. Higher derivative estimates for the second fundamental form then follow from the fact that $M_{t_0}^k$ satisfies such estimates indpendently of $k$ on compact subsets, and the Ecker--Huisken estimates. The desired convergence thus follows from the Arz\'{e}la--Ascoli theorem.

Since $D_{t_0}$ is either an interval or a half-line, $\mathbb{R} \times D_{t_0}$ is either a slab or a halfspace, both of which remain stationary under mean curvature flow by the avoidance principle, so we have $\mathbb{R} \times D_t = \mathbb{R} \times \tilde D_t = \mathbb{R} \times D_{t_0}$, and in particular $D_t = \tilde D_t = D_{t_0}$, for $t \in [{t_0}, T]$. Therefore, \cite[Theorem 1.4]{DaskSaez} implies that $\partial \Omega_t = \partial \tilde \Omega_t$ for $t \in [{t_0}, T]$, as required. 

Now suppose $n \geq 2$, and that the claim has been established for flows of dimension at most $n-1$. Arguing as above we find that the shifted flows $\partial \Omega_t - k e_{n+1}$ subconverge in $C^\infty_{\loc}(\mathbb{R}^{n+1} \times (0, T])$ to $\partial D_t \times \mathbb{R}$ as $k \to \infty$, so $\partial D_t$ is a smooth mean curvature flow of dimension $n-1$ in $\mathbb{R}^n$. By the splitting theorem, we may assume $D_t = \mathbb{R}^m \times D_t^\perp$, where $\partial D_t^\perp$ is a smooth mean curvature flow in $\mathbb{R}^{n-m}$ which is locally uniformly convex. Similarly, $\{\partial \tilde D_t\}_{t \in [{t_0}, T]}$ is a smooth mean curvature flow in $\mathbb{R}^n$, and we have $\tilde D_{t_0} = \mathbb{R}^m \times D_{t_0}^\perp$, so by the avoidance principle $\tilde D_{t}$ splits as $\mathbb{R}^m \times \tilde D_{t}^\perp$ for $t \in [{t_0}, T]$. Applying the inductive hypothesis we conclude $D_t^\perp = \tilde D_t^\perp$ for $t \in [{t_0}, T]$, and hence $D_t = \tilde D_t$ for $t \in [{t_0}, T]$. In this case \cite[Theorem 1.4]{DaskSaez} implies $\partial \Omega_t = \partial \tilde \Omega_t$ for $t \in [{t_0}, T]$.   
\end{proof}

\begin{proof}[Proof of Proposition \ref{arrival concave}]
Let $\{\partial \Omega_t\}_{t \in (-T,0]}$ be a convex mean curvature flow. If $H$ vanishes anywhere, then $\Omega_t$ is for each $t$ a stationary slab or halfspace (which has constant arrival time) so we may assume that $H>0$. If $\partial \Omega_t$ is compact, then the claim follows from 
\cite[Theorem 7.6]{ES91}, so suppose that $\partial \Omega_t$ is noncompact. We may assume without loss of generality that $A>0$, for otherwise the solution splits and we can apply the following argument to the lower dimensional factor with $A>0$. 

Fix a time $t_0 \in(-T,0)$. Let $\Omega_{t_0}^k$ be a sequence of smooth bounded convex open sets such that $\Omega_{t_0}^k \to \Omega_{t_0}$ in $C^\infty_{\loc}(\mathbb{R}^{n+1})$. For each $k$, let $\{\partial \Omega_t^k\}_{t\in[t_0, T_k)}$ be the maximal convex mean curvature flow starting from $\partial \Omega_{t_0}^k$. Let $B_{4r}$ be an open ball in $\Omega_0$. Since $\Omega_0 \subset \Omega_{t_0}$, we have $B_{4r} \subset \Omega_{t_0}$, so after passing to a subsequence we may assume $B_{2r} \subset \Omega_{t_0}^k$ for all $k$. Let 
\[T_k' \doteqdot \sup\{t \in [t_0, T_k) : B_r \subset \Omega_t^k\}\]
and let $T' \doteqdot \liminf_{k\to \infty} T_k'$. The avoidance principle ensures that $T' > t_0$. By Proposition \ref{prop:interior_est} and the Ecker--Huisken derivative estimates, after passing to a subsequence,  $\{\partial \Omega_t^k\}_{t \in [t_0, T']}$ converges in $C^\infty_{\loc}(\mathbb{R}^{n+1}\times [t_0, T'])$ to a convex mean curvature flow $\{\partial \tilde \Omega_t\}_{t \in [t_0, T']}$ with $\partial \tilde \Omega_{t_0} = \partial \Omega_{t_0}$. Applying Lemma \ref{lem:convex uniqueness}, we see that $\tilde \Omega_t = \Omega_t$ for $t \in [t_0, \min\{0,T'\}]$. If $T' \leq 0$ then $\tilde \Omega_{T'} = \Omega_{T'}$ contains $B_{4r}$, so $\Omega_T^k$ contains $B_{2r}$ for large $k$, but this contradicts the definition of $T'$, so $T' >0$.

Let $u_k : \Omega_{t_0}^k \to \mathbb{R}$ denote the arrival time of $\{\partial \Omega_t^k\}_{t \in [t_0, T_k)}$. By \cite[Theorem 2.1]{BLconcavity}, the function $\sqrt{2(u_k-t_0)}$ is concave in $\Omega_{t_0}$, and hence so too is the function $w_k :=  \min\{\sqrt{2(u_k-t_0)},\sqrt{-2t_0}\}$. Writing $u$ for the arrival time of $\{\partial \Omega_t\}_{t \in (-T,0]}$, we have that $w_k \to \sqrt{2(u -t_0)}$ in $C^0_{\loc}(\Omega_{t_0})$, and hence $\sqrt{2(u -t_0)}$ is concave in $\Omega_{t_0}$. Since $t_0 \in (-T,0)$ was arbitrary, in case $T <\infty$ we conclude that $\sqrt{2(u+T)}$ is concave. In case $T = \infty$, we conclude that $\sqrt{2(u - t_0)}$ is concave for every $t_0<0$, and from this it follows that $u$ is concave. 
\end{proof}

As a corollary, we may remove the boundedness of curvature in the hypotheses of Hamilton's differential Harnack inequality.

\begin{corollary}\label{cor:harnack}
If $\{\partial \Omega_t\}_{t\in(-T,0]}$ is a convex mean curvature flow, then
\[
\partial_t H + 2 \langle \nabla H, v \rangle + A(v, v)+\frac{H}{2(t+T)} \geq 0
\]
for every $v$ tangent to $\partial \Omega_t$.
\end{corollary}

Note that Corollary \ref{cor:harnack} can also be proven directly using an approximation by compact solutions like that in the proof of Proposition \ref{arrival concave}. This was the approach taken in \cite{DaskSaez}, which established Corollary \ref{cor:harnack} for convex entire graphs. 

\section{Sequential compactness of convex ancient solutions}\label{sec:compactness}

We now prove Theorem \ref{thm:entire_compact}. Our main tool is the following ``slab estimate''. This was proven by Wang \cite[Theorem~1.3]{Wa11} for solutions with concave arrival time, and thus holds for all convex ancient flows by Proposition \ref{arrival concave}. 

\begin{proposition}[Slab estimate]
\label{prop:slab}
There exist $\beta=\beta(n)>0$ and $R=R(n)<\infty$ with the following property. Let $\{\partial \Omega_t\}_{t\in(-\infty,0]}$ be a convex ancient mean curvature flow in $\mathbb{R}^{n+1}$ such that $0 \in \Omega_0$. If $\Omega_{-1} \cap B_R(0) \subset \{|x_{n+1}| \leq \beta\}$, then there is a fixed slab in $\mathbb{R}^{n+1}$ which contains $\Omega_t$ for all $t \leq 0$.
\end{proposition}

\begin{remark}\label{rem:Wang is a Wang2}
As alluded to in Remark \ref{rem:Wang is a Wang1}, in the initial version of this paper we were able to prove Proposition \ref{prop:slab} without exploiting concavity of the arrival time. Since this result is so crucial to our analysis, and our initial proof contains ideas\footnote{For example, Lemmas \ref{lem:width}, \ref{lem:width_2} and \ref{lem:height} are novel and provide very fine control on convex solutions contained in slab regions.} that may be useful in other contexts, we have included it as an appendix.
\end{remark}

The slab estimate implies the following ``paraboloid estimate'' \cite[Theorem~2.2]{Wa11} (we give a short proof for the convenience of the reader). In conjunction with the interior curvature estimate (Proposition \ref{prop:interior_est}), this will allow us to deduce compactness properties for sequences of convex ancient flows. 

\begin{proposition}[Paraboloid estimate]
\label{prop:paraboloid}
There exists $\eta=\eta(n)>0$ with the following property. Let $\{M_t = \partial \Omega_t\}_{t \in (-\infty,0]}$ be an entire convex ancient mean curvature flow. Given $p_0 \in M_{t_0}$,  
\[B_{\eta \sqrt{t_0-t}}(p_0) \subset \Omega_t \]
for all  $t \leq t_0- H(p_0,t_0)^{-2}$. 
\end{proposition}

\begin{proof}
Suppose, contrary to the claim, that there is a sequence of entire convex ancient solutions $\{M_t^i= \partial \Omega_t^i\}_{t \in(-\infty,0]}$ with the following properties:
\begin{itemize}
\item[--] $M_0^i$ contains the origin.
\item[--] There is a sequence of times $-\tau_i \leq -H^i(0,0)^{-2}$ such that 
\[\tau_i^{-1/2} \dist(0, M_{- \tau_i}^i)\to 0\]
as $i \to \infty$.
\end{itemize}
Performing a parabolic rescaling by $\tau_i^{-\frac{1}{2}}$ for each $i \in \mathbb{N}$, we may assume that $\tau_i=1$ and $H^i(0,0) \geq 1$. Passing to a subsequence, we may assume that $\Omega_{-1}^i$ converges locally uniformly in the Hausdorff topology to a closed convex set $K$. 

We consider two cases. First, if $K$ has no interior, then it lies in a hyperplane by convexity. We are assuming $0 \in \Omega_{-1}^i$, so $K$ contains the origin, and up to a rotation we may assume that $K \subset \{x_{n+1} = 0\}$. In particular, given any $\beta >0$ and $R<\infty$, for all sufficiently large $i$ we have 
\[\Omega_{-1}^i \cap B_{R}(0) \subset \{|x_{n+1}| \leq \beta\}.\]
Choosing $R$ sufficiently large and $\beta$ sufficiently small, by Theorem \ref{prop:slab}, this implies $M_t^i$ is contained in a slab for all $t\leq 0$, contrary to assumption. 

Suppose instead $K$ contains an open ball $B_{2\rho}$. Let $T$ be the supremum over all times $t \leq 0$ such that $\Omega_t^i$ has a subsequential Hausdorff limit containing $B_{\rho/2}$. The avoidance principle and $B_{2\rho} \subset K$ ensure that $T>-1$, and by Proposition \ref{prop:compactness}, $\{M_t^i\}_{t \in (-\infty, T]}$ subconverges in $C^\infty_{\loc}$ to a smooth convex ancient solution $\{M_t=\partial \Omega_t\}_{t \in (-\infty,T]}$.  Since $0 \in \overline\Omega{}_T^i$ and $\dist(0,M_{-1}^i)\to 0$, we have $0 \in M_t$ for all $t \in [-1, T]$. Applying the strong maximum principle to $H$ shows $M_t$ is stationary, and thus consists of a hyperplane or pair of parallel hyperplanes for all $t \leq T$. In this case $B_{2\rho} \subset K  = \overline{\Omega}_{-1}$ implies $B_{2\rho} \subset \Omega_T$, hence there is a subsequence in $i$ such that $B_{\rho} \subset  \Omega_T^i$, and unless $T=0$ we obtain a contradiction to the maximality of $T$ using the avoidance principle.  Thus, $M_0^i$ converges in $C^\infty_{\loc}$ to a hyperplane or pair of parallel hyperplanes, but we rescaled to ensure $H^i(0,0) \geq 1$, so this is impossible.
\end{proof}

Using Proposition \ref{prop:paraboloid} we now establish a convergence result for sequences of entire solutions that lose their interior at $t=0$. This  can be viewed as a complement to Proposition \ref{prop:compactness}. Indeed, there convergence of $\Omega_0^i$ in the Hausdorff topology to a convex
set $K$ of dimension $n+1$ was shown to imply convergence of the flows up to and including time zero. In the following we obtain convergence up to (but not including) time zero when $\dim K\le n$.  

\begin{proposition}\label{prop:entire_compact}
Let $\{M_t^i=\partial \Omega_t^i\}_{t \in (-\infty,0]}$ be a sequence of entire convex mean curvature flows such that $0 \in M_0^i$ and
\[ \liminf_{i \to \infty} \sup_{B_R(0) \times [-R^2,0]} H^i >0\]
for some $R < \infty$,  and suppose $\Omega_0^i$ converges in the Hausdorff topology to a closed convex set of dimension at most $n$. The sequence $\{M_t^i\}_{t \in (-\infty, 0)}$ subconverges in $C^\infty_{\loc}(\R^{n+1}\times(-\infty,0))$ to a limiting solution $\{M_t=\partial\Omega_t\}_{t \in (-\infty,0)}$ such that $\dim( \cap_{t<0} \, \Omega_t) \leq n$ and $B_{\eta\sqrt{-t}}(0) \subset \Omega_t$ for all $t < 0$, where $\eta$ is the constant from Proposition \ref{prop:paraboloid}.
\end{proposition}

\begin{proof}
There is a sequence $(x_i,t_i) \in B_R(0) \times [-R^2, 0]$ and a constant $\varepsilon >0$ such that $H^i(x_i,t_i) \geq \varepsilon$. Let us pass to a subsequence so that $(x_i,t_i)$ converges to a limit $(\bar x, \bar t)$ lying in $Q \doteqdot \overline{B_{R}(0)} \times [-R^2,0]$. Using Proposition~\ref{prop:paraboloid} we conclude 
\begin{equation}\label{eq:interior paraboloid}
B_{\eta \sqrt{t_i-t}}(x_i) \subset \Omega_t^i, \qquad t \leq t_i - \varepsilon^{-2}.
\end{equation}
In particular, for $i$ sufficiently large and $t \leq \bar t - \varepsilon^{-2} -1\doteqdot T$, 
\[B_{\eta/2}(\bar x) \subset B_{\eta \sqrt{t_i-t}}(x_i) \subset \Omega_t^i.\]

Let $u_i$ denote the arrival time of $\Omega_t^i$, and write $K_i \doteqdot \{(x,t) \in \mathbb{R}^{n+1} \times (-\infty,0] : t \leq u_i(x)\}$. By Proposition \ref{arrival concave}, $K_i$ is convex. Passing to a further subsequence if necessary, we may assume $K_i$ converges to a closed convex set $K$ in the Hausdorff topology. Since $0 \in \partial \Omega_0^i$, we have $0 \in K$. We also know $K$ contains the horizontal disc $B_{\eta/ 2}(\bar x) \times \{T\}$, so $K$ contains the convex hull of this disc and 0. In particular, $K_t := \{ x \in \mathbb{R}^n : (x,t) \in K\}$ has nonempty interior for all $t <0$. By Proposition \ref{prop:compactness}, we may pass to a subsequence such that $\{\partial \Omega_t^i\}_{t \in (-\infty,0]}$ converges in $C^\infty_{\loc}(\mathbb{R}^{n+1} \times (-\infty,0))$ to a limiting flow $\{\partial \Omega_t\}_{t \in (-\infty,0)}$. By assumption the dimension of $K_0 = \cap_{t<0}\,\Omega_t$ is at most $n$. 

Passing to a further subsequence, we may assume $H_i(0,0)$ converges to $H_0 >0$. Indeed, $\liminf_{i \to \infty} H_i(0,0) = 0$ then $K$ has a supporting halfspace at the origin which is vertical, but this contradicts \eqref{eq:interior paraboloid}. It now follows that there is a subsequence of times $t_i \to 0$ and points $y_i \in \partial \Omega_{t_i}^i$ such that $H_i(y_i,t_i) \to \infty$. Indeed, if there is no such sequence Proposition \ref{prop:compactness} implies the sets $\Omega_0^i$ are converging smoothly to an open convex set which thus has dimension $n+1$. Appealing to Proposition \ref{prop:paraboloid}, we conclude that $B_{\eta \sqrt{-t}}(0) \subset \Omega_t$ for all $t <0$. 
\end{proof}

In fact, it is possible to say more about the limit solution arising in Proposition \ref{prop:entire_compact}. Indeed, arguing as in \cite[Lemma 2.9]{Wa11}, we obtain the following.
\begin{proposition}
\label{prop:entire_compact_additional_material}
Let $\{M_t = \partial\Omega_t\}_{t \in (-\infty,0)}$ be the limit flow obtained in Proposition \ref{prop:entire_compact}. The limit set $\Sigma \doteqdot\cap_{t<0} \, \Omega_t$ is a linear subspace of $\mathbb{R}^{n+1}$. Moreover, for each $t<0$, 
\[M_t = \Sigma \times M_t^\perp\]
where $\{M_t^\perp\}_{t\in(-\infty,0)}$ is a compact convex solution of mean curvature flow in $\Sigma^\perp$.
\end{proposition}
\begin{proof}
If $n=1$ the claim is an immediate consequence of the classification in \cite{BLT3}, so assume $n\geq 2$.

Suppose first that $K\doteqdot \cap_{t<0}\,\Omega_t$ does not contain any lines. In this case, we claim $K$ consists of the origin (note that $0\in K$ since $B_{\eta \sqrt{-t}}(0) \subset \Omega_t$). It suffices to prove that $K$ is compact, since then $\Omega_t$ is compact for all $t <0$, and hence $K$ is a point by Huisken's theorem \cite{Hu84}. 

So suppose, to the contrary, that $K$ is noncompact. Since $K$ does not contain any lines, we may choose Euclidean coordinates for $\mathbb{R}^{n+1}$ such that $\mathbb{R}_+ e_1 \subset K$ and $K \cap \{x_1 = c\}$ is compact for all $c \geq 0$. In particular, $\tilde M_t \doteqdot M_t \cap \{x_1 = 0\}$ is compact for all $t<0$. Fix a sequence of displacements $R_i \to \infty$ and define, for each $t\le 0$, 
\[M_t^i \doteqdot  M_t - R_i e_1\,.\]
Using Proposition \ref{prop:compactness} and the avoidance principle, we conclude that a subsequence of the flows $\{M_t^i\}_{t\in(-\infty,0)}$ converges in $C^\infty_{\loc}$ to a limiting flow $\{M_t'=\partial \Omega_t'\}_{t \in(-\infty,0)}$ which satisfies $\dim(\cap_{t<0} \, \Omega_t')\leq n$. Moreover, $\mathbb{R} e_1 \subset \Omega_t'$ for all $t<0$, and hence $M_t'$ splits off $\mathbb{R} e_1$. 

Observe that $\Omega_t \subset \Omega_t'$ for all $t <0$. In particular, $\tilde M_{-1}$ is contained in the closure of $\Omega_{-1}' \cap \{x_1 = 0\}$. Since the family $M_t' \cap \{x_1=0\}$ solves mean curvature flow in $\{x_1=0\}$, Lemma \ref{lem:slice_subsols} and the avoidance principle imply that the curves $\tilde M_t$ and $M_t' \cap \{x_1= 0\}$ either coincide, or else the distance between them is positive and nondecreasing in time. However we also know that $\tilde M_t$ and $M_t' \cap \{x_1= 0\}$ both reach the origin at $t=0$, so in fact the two families of curves coincide. In particular, $M_t$ and $M_t'$ make interior contact for all $t<0$, and therefore coincide by the strong maximum principle. This is a contradiction, since $\Omega_t'$ contains $\mathbb{R} e_1$, but we assumed $K$ does not contain any lines. 

This completes the case where $K$ does not contain any lines. In general, we may write $M_t = \Sigma \times \partial \Omega_t^\perp$ for all $t < 0$, where $\Sigma$ is a fixed affine subspace and $\cap_{t<0} \, \Omega_t^\perp$ does not contain any lines. Repeating the above argument we find that $\cap_{t<0} \, \Omega_t^\perp = \{0\}$, and hence $\cap_{t<0} \, \Omega_t = \Sigma$. 
\end{proof}

We now establish the key compactness results for sequences of entire flows. 

\begin{theorem}[Sequential compactness of entire flows]
\label{thm:entire_compact}
Let $\{\partial \Omega_t^i \}_{t \in (-\infty,0]}$, $i\in\N$, constitute a sequence of entire convex ancient mean curvature flows. Suppose that $0 \in \partial\Omega_0^i$ for each $i$ and $H^i(0,0)\to H_0\in [0,\infty]$ as $i\to\infty$. After passing to a subsequence, $\{\partial\Omega_t^i\}_{t\in(-\infty,0]}$ converges in $C^\infty_{\loc}(\R^{n+1}\times(-\infty,0))$ to a convex ancient mean curvature flow $\{\partial \Omega_t\}_{t \in (-\infty,0)}$. Moreover,
\begin{enumerate}
\item \label{item:entire_compact_0} if $H_0=0$, then the convergence is in $C^\infty_{\loc}(\R^{n+1}\times(-\infty,0])$ and the limit is a stationary hyperplane of multiplicity one.
\item \label{item:entire_compact_1} if $H_0\in(0,\infty)$, then the convergence is in $C^\infty_{\loc}(\R^{n+1}\times(-\infty,0])$ and the limit $\{\partial \Omega_t\}_{t \in (-\infty,0]}$ is entire. 
\item \label{item:entire_compact_2} if $H_0=\infty$, then the limit $\{\partial \Omega_t\}_{t \in (-\infty,0)}$ is entire, $\Sigma\doteqdot \cap_{t<0}\, \Omega_t$ is an affine subspace of $\R^{n+1}$, and $\Omega_t$ splits as a product $\Sigma\times\Omega^\perp_t$, where $\{\Omega^\perp_t\}_{t\in(-\infty,0)}$ is a family of bounded convex bodies in $\Sigma^\perp$.
\end{enumerate}
\end{theorem}

\begin{proof}
Let $u_i : \mathbb{R}^{n+1} \to \mathbb{R}$ be the arrival time of $\{\partial \Omega_t^i\}_{t \in (-\infty,0]}$, which is concave by Proposition \ref{arrival concave}. Let $K_i$ denote the closed convex region under the graph of $u_i$, i.e.,
\[K_i := \{(x,t) \in \mathbb{R}^{n+1}\times (-\infty,0]: t \leq u_i(x)\}.\]
By assumption, $0 \in K_i$, so after passing to a subsequence we may assume that the $K_i$ converge in the Hausdorff topology to a closed convex set $K$ with $0 \in K$. Let $K_t := \{x \in \mathbb{R}^n : (x,t)\in K\}$. 

Consider first the case that $H_0=0$. This implies that $K$ admits a vertical supporting hyperplane at the origin. If $\dim(K_0)=n+1$, then Proposition \ref{prop:compactness} implies that the convergence is  smooth up to time zero, and hence the limit must satisfy $H=0$ at the spacetime origin. The strong maximum principle then implies that $H\equiv 0$, so $\Omega_t$ is a stationary slab or halfspace. If $\Omega_t$ is a slab then, given any $R<\infty$ and $\beta >0$, for large $i$ we may rescale $\Omega_t^i$ so that $\Omega_{-1}^i \cap B_R(0) \subset \{|x_{n+1}| \leq \beta\}$. By Proposition \ref{prop:slab}, this implies $\Omega_t^i$ lies in a fixed slab for all $t \leq 0$, contrary to assumption. Therefore, $\Omega_t$ is a halfspace, which establishes the claim. 

So suppose that $\dim(K_0)\le n$. If
\[
\liminf_{i \to \infty} \sup_{B_R(0) \times [-R^2,0]}H^i>0
\]
for some $R>0$, then Proposition \ref{prop:entire_compact} implies that 
$B_{\eta\sqrt{-t}}\subset K_t$. But this is impossible, since $K$ admits a vertical supporting hyperplane at the origin. 

So in fact
\[
\liminf_{i \to \infty} \sup_{B_R(0) \times [-R^2,0]}H^i=0
\]
for all $R>0$. But then the Ecker--Huisken interior estimates imply that, after passing to a subsequence, the convergence is smooth up to time zero and the limit satisfies $H\equiv 0$. This implies that the limit is a hyperplane of multiplicity one or two. The latter is ruled out by Proposition \ref{prop:slab}. 

Suppose next that $H_0 \in (0, \infty)$. In this case $K$ admits a supporting halfspace $S$ of slope $H_0^{-1}$ at the origin. It follows that $\dim(K_0) = n+1$, since if $\dim(K_0) \leq n$ then Proposition \ref{prop:paraboloid} implies that
\[\{(x,t) : t \leq \eta^{-2} |x|^2 \} \subset K \subset S,\]
which contradicts the fact that $S$ is not horizontal. We conclude that $K_0$ has nonempty interior, and thus deduce part (2) of the theorem from Proposition \ref{prop:compactness}.

Next, suppose $H_0 = \infty$. In this case Proposition \ref{prop:compactness} implies that $K_0$ has no interior, and hence part \eqref{item:entire_compact_2} of the theorem follows immediately from Proposition \ref{prop:entire_compact} and Proposition \ref{prop:entire_compact_additional_material}. 
\end{proof}

Theorem \ref{thm:noncollapsing equivalence} now follows immediately.

\begin{proof}[Proof of Theorem \ref{thm:noncollapsing equivalence}]
It is easy to deduce that (3) implies (1). The implication (2) implies (3) follows from the strong maximum principle as in \cite{MR3384488}. Thus it suffices to show that (1) implies (2). 

If this is not the case, then there is a sequence of convex ancient mean curvature flows $\{\partial \Omega_t^i\}_{t \in (-\infty,0]}$ such that $0 \in \partial \Omega_0^i$ and $H_i(0,0) =1$, but the inscribed radius $\bar r_i$ at (0,0) satisfies $\bar r_i \to 0$. Part (2) of Theorem \ref{thm:entire_compact} tells us that $\{\partial \Omega_t^i\}_{t \in (-\infty,0]}$ subconverges in $C^\infty_{\loc}(\mathbb{R}^{n+1} \times (-\infty,0])$ to a convex ancient mean curvature flow $\{\partial \Omega_t\}_{t \in (-\infty,0]}$. In particular, $\liminf_{i \to \infty} \bar r_i >0$, which is a contradiction.
\end{proof}

\section{Grim Reapers at infinity}
\label{sec:Reapers}

We next prove that every convex ancient mean curvature flow which is contained in a slab is asymptotic to at least one Grim hyperplane. To begin, we show that every such solution actually sweeps out a slab. 

\begin{lemma}
\label{lem:slabsweep}
Let $\{M_t=\partial \Omega_t\}_{t \in (-\infty,0]}$ be a convex ancient mean curvature flow such that $M_t \subset \{|x_{n+1}| \leq \beta' \}$ for all $t \leq 0$. There are constants $\beta \leq \beta'$ and $a \in \mathbb{R}$ such that 
\[\cup_{t \leq 0} \, \Omega_t = \{|x_{n+1}-a|\leq\beta\}.\]
\end{lemma}
\begin{proof}
We may assume without loss of generality $0 \in \Omega_0$. The claim can be restated as follows: every supporting halfspace of $\cup_{t \leq 0}\,\Omega_t$ is of the form $\{\pm\, x\cdot e_{n+1} \leq c\}$. We prove this statement by contradiction. 

Suppose $e$ is a unit vector such that $|e \cdot e_{n+1}| < 1$, yet 
\[\cup_{t\leq 0}\, \Omega_t \subset \{x \cdot e \leq c\}.\]
Let $\Sigma \doteqdot \spa\{e, e_{n+1}\}$ and define 
\[\tilde M_t \doteqdot M_t \cap \Sigma.\]
According to Lemma~\ref{lem:slice_subsols} the family of smooth convex curves $\{\tilde M_t\}_{t\in(-\infty, 0]}$ form a subsolution of curve-shortening flow in $\Sigma$. On the other hand, since
\[\tilde M_t \subset \Sigma \cap \{|x_{n+1}| \leq \beta'\} \cap \{x \cdot e \leq c\},\]
given any $\tau \leq 0$ we can enclose $\{\tilde M_t\}_{t\in [\tau, 0]}$ in a Grim Reaper (asymptotic to, say, the pair of lines $\Sigma \cap \{|x_{n+1}| = 2 \beta'\}$) which reaches the origin at time $t = \tau+ T$, where $T  < \infty$ depends only on $\beta'$, $e$ and $c$. We arranged that $0 \in \Omega_0$, so for $|\tau|$ sufficiently large we obtain a contradiction to the avoidance principle.
\end{proof}

We may now prove the Grim Reaper detection theorem. This is inspired by \cite[Proposition 5.6]{BLTatomic}, but here we do not assume $|A|$ is bounded on compact time intervals. 

\begin{proposition}[Grim Reapers at infinity]
\label{prop:grim_plane}
A convex ancient solution $\{M_t\}_{t \in(-\infty,0]}$ satisfying $M_t \subset \{|x_{n+1}| \leq \beta\}$ for all $t \leq 0$ admits a sequence of rescaled spacetime translations which converges in $C^\infty_{\loc}(\R^{n+1}\times(-\infty,\infty))$ to $\{\mathbb{R}^{n-1} \times \Gamma_t\}_{t \in (-\infty,\infty)}$ after a fixed rotation in  $\R^{n+1}$, where $\{\Gamma_t\}_{t\in(-\infty,\infty)}$ is the Grim Reaper.
\end{proposition} 

For the convenience of the reader we note the following proof simplifies considerably in case $M_t$ is reflection symmetric through the middle plane $\{x_{n+1}=0\}$ for all $t\leq 0$.  Under this additional assumption $\tilde \Omega_t$ is a subset of $\Omega_t$, so it suffices to take $u_{k,i}= \tilde u_{k,i}$ and $v_{k,i}=\tilde v_{k,i}$ in the argument below.

\begin{proof}[Proof of Proposition \ref{prop:grim_plane}]
In light of Lemma \ref{lem:slabsweep} we may assume without loss of generality
\[\cup_{t \leq 0} \, \Omega_t = \{|x_{n+1}| \leq \beta\}.\]
Let $\pi$ denote the orthogonal projection of $\mathbb{R}^{n+1}$ onto $\{x_{n+1} = 0\}$, define $\tilde \Omega_t \doteqdot \pi(\Omega_t)$, and let $\tilde M_t$ denote the $(n-1)$-dimensional boundary of $\tilde \Omega_t$ in $\{x_{n+1} = 0\}$. Observe $p \in M_t$ satisfies $\pi(p) \in \tilde M_t$ if and only if $\nu(p,t) \cdot e_{n+1} = 0$, and we have 
\begin{equation}
\label{eq:slice_sweep}
\cup_{t \leq 0}\,\tilde \Omega_t = \{x_{n+1}=0\}.
\end{equation}

Fix a sequence $t_i \to -\infty$ and let $\tilde y_i \in \tilde M_{t_i}$ be such that 
\[|\tilde y_i| = \min_{y \in \tilde M_{t_i}} |y|.\]
There is a corresponding sequence $y_i \in M_{t_i}$ such that $\pi(y_i) = \tilde y_i$, hence $\nu(y_i, t_i) \cdot e_{n+1} = 0$. Define 
\[\varepsilon' \doteqdot \limsup_{i \to \infty} \sup_{B_{4\beta}(y_i)} \frac{H(\cdot, t_i)}{2}\qquad \text{and} \qquad \varepsilon \doteqdot  \min\{1 , \varepsilon'\}.\]
If $\varepsilon = 0$ the connected component of $(M_{t_i}-y_i)\cap B_{4\beta}(0)$ containing the origin converges to a piece hyperplane orthogonal to $\{x_{n+1}=0\}$ as $i \to \infty$. 
This contradicts $M_{t_i} \subset \{|x_{n+1}| \leq \beta\}$, so we conclude $\varepsilon >0$. Passing to a subsequence if necessary we can choose a new sequence $q_i \in M_{t_i} \cap B_{4\beta}(y_i)$ such that $H(q_i, t_i) \geq \varepsilon$.  Define $L \doteqdot 8 \varepsilon^{-1} + 8\beta$.  

It follows from \eqref{eq:slice_sweep} that the sequence of $n$-balls 
\[
\tilde B_i \doteqdot B_{|\tilde y_i|}(0) \cap \{x_{n+1} = 0\}
\]
sweeps out $\{x_{n+1} = 0\}$ as $i \to \infty$. For each $i \in \mathbb{N}$ let 
\[
\tilde P_i \doteqdot  \{x \in \mathbb{R}^{n+1}: x_{n+1} = 0, \; x \cdot \tfrac{\tilde y_i}{|\tilde y_i|} = |\tilde y_i| - L\}.
\]
Then $\dist(\tilde P_i, \tilde y_i) = L$, and the sequence of $(n-1)$-balls $\tilde P_i \cap \tilde B_i$ sweeps out $\tilde P_i$ as $i \to \infty$. To ease notation, for each $i \in \mathbb{N}$ let us perform a rotation (leaving $e_{n+1}$ fixed) and translation of coordinates so that $\tilde y_i$ coincides with $L e_n$. Then, in the new coordinates, $\tilde P_i$ passes through the origin and is spanned by the vectors $\{e_1, \dots, e_{n-1}\}$, and $\tilde B_i$ is given by
\[\tilde B_i = B_{|\tilde y_i|}(-(|\tilde y_i| - L) e_n)\cap \{x_{n+1} = 0\}.\]
Moreover, for each $1 \leq k \leq n-1$, the points 
\[\tilde u_{k,i} \doteqdot \tilde B_i \cap  \mathbb{R}_- e_k, \qquad \tilde v_{k,i} \doteqdot \tilde B_i \cap \mathbb{R}_+ e_k\]
satisfy $|\tilde u_{k,i}| = |\tilde v_{k,i}|\to \infty$ as $i \to \infty$. More explicitly, we have 
\[|\tilde u_{k,i} |^2 = |\tilde v_{k,i}|^2 = |\tilde y_i|^2 - (|\tilde y_i|-L)^2= 2L|\tilde y_i| - L^2.\]

Since $\tilde u_{k,i}$ and $\tilde v_{k,i}$ are in $\tilde B_i \subset \tilde \Omega_{t_i}$ there are corresponding points $u_{k,i}$ and $v_{k,i}$ in $\Omega_{t_i}$ such that
\[\pi(u_{k,i}) = \tilde u_{k,i}, \qquad \pi(v_{k,i}) = \tilde v_{k,i}.\]
Let $\rho_{k,i}$ denote the segment joining $u_{k,i}$ to $v_{k,i}$. Then since $\Omega_{t_i}$ is convex we certainly have 
\[\rho_{k,i} \subset \Omega_{t_i}.\]
Also, since $\pi(\rho_{k,i}) \to \mathbb{R} e_k$ as $i \to \infty$, and $\rho_{k,i} \subset \{|x_{n+1}| \leq \beta\}$, $\rho_{k,i}$ converges (after passing to a subsequence) as $i \to \infty$ to a line $\rho_k=b_k e_{n+1} + \mathbb{R}e_k$, where $|b_k| \leq \beta$. In particular,
\[
\theta_{k,i} \doteqdot \frac{v_{k,i} - u_{k,i}}{|v_{k,i} - u_{k,i}|} \cdot e_{n+1} \to 0, \qquad i \to \infty.
\]

Appealing to Lemma \ref{lem:pp} we find a sequence of times $\tau_i \leq 0$ and points $p_i \in M_{\tau_i}$ with the following properties:
\begin{itemize}
 \item[--] $(p_i, \tau_i) \in B_{2/\varepsilon}(q_i) \times [t_i- 2/\varepsilon^{2}, t_i]$.
\item[--] $H(p_i, \tau_i )\geq \varepsilon$. 
\item[--] $H(p,t) \leq 2 H(p_i,\tau_i)$ for all $(p,t) \in B_{r_i}(p_i) \times [\tau_i - r_i^2, \tau_i]$, where $r_i \doteqdot H(p_i,\tau_i)^{-1}$. 
\end{itemize}
Let $T_{k,i} \subset \Omega_{\tau_i}$ denote the convex hull of $\{p_i, u_{k,i}, v_{k,i}\}$, and write $z_{k,i}$ for the unique point in $\rho_{k,i}\cap \mathbb{R}e_{n+1}$.  
Since we have
\begin{align*}
|p_i-z_{k,i}|&\leq |p_i - q_i| + |q_i - y_i| + |y_i - z_{k,i}|\\
& \leq 2\varepsilon^{-1} + 4\beta + \sqrt{L^2 + 4\beta^2} < \infty,
\end{align*}
and $\rho_{k,i} \to \rho_k$, as $i \to \infty$,
\[\alpha_{k,i}\doteqdot \frac{u_{k,i} - p_i}{|u_{k,i}-p_i|} \cdot \frac{v_{k,i} - p_i}{|v_{k,i} - p_i|} \to -1.\]

We are finally ready to rescale. Consider the sequence of convex ancient flows $\{M^i_t\}_{t\in(-\infty,-\tau_ir_i^{-2}]}$ defined by
\[
M_t^i \doteqdot r_i^{-1}(M_{\tau_i + r_i^2 t} - p_i)\,.
\]
Since $H^i(0,0) = 1$ and $H^i \leq 2$ in $B_1(0) \times [-1,0 ]$, Proposition \ref{prop:compactness} implies that $\{M_t^i\}_{t \in (-\infty, 0]}$ subconverges in $C^\infty_{\loc}(\mathbb{R}^{n+1} \times (-\infty,0])$ to a limiting solution $\{M_t'\}_{t \in (-\infty,0]}$.

Since the region $\Omega_0^i$ enclosed by $M^i_0$ contains $r_i^{-1} (T_{k,i} - p_i)$ for all $1 \leq k \leq n-1$ and $i \in \mathbb{N}$, the region $\Omega_0'$ enclosed by $M'_0$ contains each of the sets 
\[T_k \doteqdot \lim_{i \to \infty} r_i^{-1} (T_{k,i} - p_i).\] In particular, since $\theta_{k,i} \to 0$ and $\alpha_{k,i} \to -1$ (note each of these quantities is scaling-invariant), and the rescaling factors satisfy $r_i^{-1} \geq \varepsilon$, we have
\[\mathbb{R}e_k \subset T_k \subset \Omega_0'\,.\]
It follows that $\Omega_t'$ splits off the line $\mathbb{R}e_k$.  
This is true for all $1 \leq k \leq n-1$, so 
\[M_t' = \mathbb{R}^{n-1}\times \Gamma_t\]
for all $t \leq 0$, where $\{\Gamma_t\}_{t \in (-\infty,0]}$ is a convex ancient solution of curve-shortening flow in $\mathbb{R}^2$. 

Noting that $\tilde \Omega_{t_i}$ contains the point $-|\tilde y_i| e_n$, the same kind of reasoning we used to show that $\mathbb{R} e_k \subset \Omega_0'$ can be employed to obtain $\mathbb{R}_- e_n \subset \Omega_0'$. In particular, $\Gamma_0$ encloses $\mathbb{R}_- e_n$, and we conclude from the classification of convex ancient solutions to curve shortening flow \cite{BLT3} that $\{\Gamma_t\}_{t\in(-\infty,0]}$ is the Grim Reaper.

We now extend the convergence to $t \in (-\infty, \infty)$. Let $B_{2r}$ be an open ball in $\Omega_0'$, and define
\[
\bar t_i \doteqdot \sup\{t\in(-\infty, -\tau_i r_i^{-2}]: B_{r} \subset \Omega_t^i\}, \qquad \bar t \doteqdot \liminf_{i \to \infty} \bar t_i.
\]
Recall that $-\tau_i r_i^{-2} \to \infty$. Suppose for a contradiction that $\bar t < \infty$. Using the avoidance principle and Proposition \ref{prop:compactness}, we can pass to a further subsequence such that $\Omega_t^i$ subconverges in $C^\infty_{\loc}(\mathbb{R}^{n+1} \times (-\infty, \bar t\,])$. We know that for all $t \leq \bar t$, the limit $M_t'$ is the product of $\mathbb{R}^{n-1}$ with a convex ancient solution of curve shortening flow which is noncompact, and hence a Grim Reaper by \cite{BLT3}. That is, $M_t' = \mathbb{R}^{n-1} \times \Gamma_{t-\bar t}$ for all $t \leq \bar t$. In particular, $B_{2r} \subset \Omega_{\bar t}'$, but this contradicts the avoidance principle and definition of $\bar t$. We thus conclude that $\bar t = \infty$, in which case Proposition \ref{prop:compactness} can be used to extract a subsequence such that $M_t^i$ converges to a Grim hyperplane in $C^\infty_{\loc}(\mathbb{R}^{n-1} \times (-\infty,\infty))$. 
\end{proof}

To conclude this section we prove Theorem \ref{thm:Reapers}. 

\begin{proof}[Proof of Theorem \ref{thm:Reapers}]
By Theorem \ref{thm:noncollapsing equivalence}, \eqref{cond:collapsing} implies \eqref{cond:not entire}. Corollary 2.2 in \cite{Wa11} asserts that \eqref{cond:not entire} implies \eqref{cond:slab}.

Let $M_t = \partial \Omega_t$. If \eqref{cond:slab} holds, for every fixed $t <0$, a barrier argument as in Lemma \ref{lem:slabsweep} shows that if $r_i \to \infty$ then  $r_i \Omega_{r_i^{-2} t}$ converges to a hyperplane in the Hausdorff topology. Therefore, by the Ecker--Huisken interior estimates for graphs and the Arzel\`a--Ascoli theorem, $\{r_i^{-1} M_{r_i^2 t}\}_{t\in(-\infty,0]}$ converges locally uniformly in the smooth topology to a multiplicity-two hyperplane. So \eqref{cond:slab} implies \eqref{cond:some blowdown}. 

On the other hand, if $\{r_i^{-1} M_{r_i^2 t}\}_{t\in(-\infty,0]}$ converges to a multiplicity-two hyperplane as $r_i \to \infty$, Theorem \ref{prop:slab} implies that $\{r_i^{-1} M_{r_i^2 t}\}_{t\in(-\infty,0]}$ lies in a slab for all sufficiently large $i$. Appealing to Lemma \ref{lem:slabsweep} we conclude that \eqref{cond:some blowdown} implies \eqref{cond:slab}. 

We showed in Proposition \ref{prop:grim_plane} that \eqref{cond:slab} implies \eqref{cond:Grim}, and this implies \eqref{cond:collapsing} since the Grim Reaper is collapsing. This completes the proof.
\end{proof}

\section{Singularity formation in immersed or high codimension flows}\label{sec:singularities}

Finally, we show that curvature pinching can be used to rule out collapsing singularity models.

\begin{proof}[Proof of Corollary \ref{cor:n-1 convex}]
The Huisken--Sinestrari convexity estimate \cite{HuSi99b} implies that the second fundamental form $A$ is nonnegative, so we have $|A|^2 \leq H^2$. If $|A|^2 = H^2$ at an interior spacetime point then, applying the strong maximum principle to the inequality
\[(\partial_t - \Delta) \frac{|A|^2}{H^2} \leq \frac{2}{H} \left\langle \nabla H, \nabla \frac{|A|^2}{H^2} \right \rangle, \] 
we conclude that $|A|^2 \equiv H^2$, which is only possible if all but one of the principal curvatures vanish identically. This is ruled out by uniform $(n-1)$-convexity, so $|A|^2 < H^2$. Sacksteder's theorem \cite{Sacksteder} now implies that $M_t$ bounds some convex body $\Omega_t$ for each $t\leq0$.

Uniform $(n-1)$-convexity rules out asymptotic Grim hyperplanes, so, by Theorem \ref{thm:Reapers}, the domains $\Omega_t$ sweep out $\mathbb{R}^{n+1}$. Theorem \ref{thm:noncollapsing equivalence} now implies the desired noncollapsing property.
\end{proof}

The proof of Corollary \ref{cor:pinched high codim} is analogous. 
\begin{proof}[Proof of Corollary \ref{cor:pinched high codim}]
Under the quadratic pinching hypothesis, which is preserved by mean curvature flow \cite{AnBa10}, blow-ups are codimension one \cite{NaffPlanarity} and convex \cite{LynchNguyenConvexity}. In codimension one, the inequality $|A|^2 \leq \tfrac{4}{3n} |H|^2$ implies uniform $(n-1)$-convexity, so the claim follows as above.
\end{proof}

\appendix

\section{Proof of the slab estimate}

We include a proof of X.-J. Wang's slab estimate (Propsition \ref{prop:slab} of the present paper) for convex ancient mean curvature flows. The argument follows the scheme outlined by Wang, but differs in that it does not make use of concavity of the arrival time.  

For the remainder of this section we consider a fixed convex ancient mean curvature flow $\{M_t\}_{t \in (-\infty,0]}$, $M_t = \partial \Omega_t$, such that $0 \in \Omega_0$. We assume that $M_t$ is noncompact and of dimension $n \geq 2$ (although with minor modifications our arguments apply also to compact solutions and solutions with $n=1$). In light of the splitting theorem for convex solutions (see, for example, \cite[Theorem 9.11]{EGF}), we need only consider the case where $M_t$ is locally uniformly convex. In particular, $\Omega_0$ does not contain any lines, so we may choose an orthonormal basis $\{e_i\}$ for $\mathbb{R}^{n+1}$ such that, for all $t \leq 0$,
\begin{itemize}
\item[--] $\Omega_t$ contains $\mathbb{R}_- e_1 \doteqdot \{s e_1 : s \leq 0\}$.
\item[--] $M_t \cap \{x_{1} = 0\}$ is compact.
\end{itemize}

We introduce the notation
\[\mathcal M \doteqdot \{(p, \tau) \in \mathbb{R}^{n+1} \times [0,\infty): p \in M_{-\tau}\}\]
and observe that $\mathcal M = \mathcal M_+ \cup \mathcal M_-$, where the pieces $\mathcal M_+$ and $\mathcal M_-$ are characterised by the inequalities $\nu \cdot e_{n+1} \geq 0$ and $\nu\cdot e_{n+1} \leq 0$, respectively. Where there is no chance of confusion we write $\mathbb{R}^n$ for $\spa\{e_1, \dots, e_n\}$. If $U$ is the orthogonal projection of $\mathcal M$ onto $\mathbb{R}^n \times [0,\infty)$, then there are functions $g_{\pm}:U \to \mathbb{R}$ such that 
\[\mathcal M_+ = \Gaph g_+, \qquad \mathcal M_- = \Gaph g_-.\]
Since $\Omega_t$ is convex, for all $t\le 0$, $g_+(\cdot, \tau)$ is concave and $g_-(\cdot, \tau)$ convex for each $\tau \geq 0$. Hence $g \doteqdot g_+ - g_-$ is concave in space at each $\tau \geq 0$. We write $U_\tau \doteqdot \{x \in \mathbb{R}^n : (x,\tau) \in U\}$ and $B^n_r(x) \doteqdot \{y \in \mathbb{R}^n: |y-x|<r\}$. Where possible, we use $p$ and $q$ to denote points in $\mathbb{R}^{n+1}$, with $x$ and $y$ denoting points in $\mathbb{R}^n$. 

\subsection{Horizontal displacement bounds} 

Our first estimate is a bound for the horizontal displacement, which asserts that over time intervals where $g(0,\tau)$ and $Dg_{\pm}(0,\cdot)$ are controlled, the displacement of $M_{-\tau}$ from the origin grows linearly in $\tau$ in directions approximately orthogonal to $e_{n+1}$.
%
%

\begin{lemma}
\label{lem:width}
Fix $\tau_1 > \tau_0 $. If $g(0,\tau_1) \leq C$ and $|D g_{\pm} (0, \tau)|\leq 1$ for all $\tau \in [\tau_0, \tau_1]$, 
then
\[g_-(x, \tau_1) \leq b + v\cdot x \leq g_+(x, \tau_1)\;\;\text{for all}\;\;x \in B^n_{\mu (\tau_1 - \tau_0)/C}(0)\,,\]
where $\mu = \mu(n)$, $b\doteqdot\tfrac{1}{2}( g_+(0, \tau_1) + g_-(0, \tau_1))$, and $v\doteqdot\sum_{k=2}^n D_k g_+ (0, \tau_1) e_k$.
\end{lemma}
\begin{proof} 
First, let us relabel axes so that $D_k g(0,\tau_1) \leq 0$ for all $2 \leq k \leq n$. Define
\[\tilde g_{\pm}(x,\tau) \doteqdot g_{\pm}(x,\tau) - v \cdot x, \]
where $v\doteqdot \sum_{k=2}^n D_k g_+ (0, \tau_1) e_k$, so that
\[D_k \tilde g_+(0, \tau_1) = 0\;\;\text{for}\;\; 2 \leq k \leq n.\]

Fix $k\in\{2,\dots, n\}$ and set $\Sigma \doteqdot \spa\{e_k, e_{n+1}\}$. We are assuming $0 \in \Omega_0$, so $M_t \cap \Sigma$ bounds an open convex region in $\Sigma$ for all $t \leq 0$. By Lemma \ref{lem:slice_subsols} the curves $\tilde M_t \doteqdot M_t \cap \Sigma$ are smooth and move inward with normal speed no less than their curvature. Recall we chose $e_1$ to ensure $M_t \cap \{x_1 = 0\}$ is compact for all $t\leq 0$, so $\tilde M_t$ is compact. Writing $\tilde \nu$ for the outward unit normal to $\tilde M_t$ in $\Sigma$, the derivative bounds $|D_k g_{\pm}(0,\tau)| \leq 1$ imply
\[\tilde \nu(g_+(0,t) e_{n+1}, t) \cdot e_{n+1} \geq \frac{1}{\sqrt{2}}\;\;\text{and}\;\;\tilde \nu(g_-(0,t) e_{n+1}, t) \cdot e_{n+1} \leq - \frac{1}{\sqrt{2}}\]
for $-\tau_1 \leq t \leq - \tau_0$. Therefore, if we let $\alpha_{\pm}(t)$ denote the area enclosed by $\tilde M_t \cap \{\pm p \cdot e_k \geq 0\}$, we have $\dot \alpha_{\pm} \leq - \tfrac{\pi}{2}$ for $-\tau_1 \leq t \leq - \tau_0$, hence 
\begin{equation}
\label{eq:width_area}
\alpha_{\pm}(-\tau_1) \geq \alpha_{\pm}(-\tau_0) + \frac{\pi}{2} (\tau_1 - \tau_0) \geq \frac{\pi}{2} (\tau_1 - \tau_0) \doteqdot  L.
\end{equation}

Define
\[
\rho\doteqdot \sup_{p \in \tilde M_{-\tau_1}} p \cdot e_k\,\,\text{ and }\,\,- \sigma \doteqdot \inf_{p \in \tilde M_{-\tau_1}} p \cdot e_k\,.
\]
Set $B=g(0, \tau_1)-\s D_kg(0,\tau_1)$. Since $D_k g(0, \tau_1)\le 0$, we find that $B\ge g(0, \tau_1)$ and, by enclosing $\tilde M_t \cap \{ p \cdot e_k \geq 0\}$ with a suitable rectangle,
\[
\a_+\le g(0, \tau_1)\rho\,.
\]
Moreover, since $\s\ge \frac{B-g(0, \tau_1)}{g(0, \tau_1)}\rho$, we either have $\s\ge \rho$ or
\[
\a_-\le \frac{B+g(0, \tau_1)}{2}\sigma\le\frac{3}{2} g(0, \tau_1)\s\,.
\]
We conclude that $\rho, \s\ge \frac{L}{2 C}$.

We claim that
\begin{equation}\label{eq:horizontal displacement first case}
\tilde g_+(se_k, \tau_1) \geq b\;\;\text{for}\;\; 
s \in [0, \tfrac{L}{4C}]\,,
\end{equation}
where $b \doteqdot \tfrac{1}{2}( g_+(0, \tau_1) + g_-(0, \tau_1))$.

Indeed, since the function $s\mapsto \tilde g_+(se_k,\tau_1)$ is concave, for $0 \leq s \leq \rho$ we may estimate
\bann
\tilde g_+(se_k,\tau_1)\ge{}& \left(1-\frac s\rho\right)\tilde g_+(0,\tau_1)+\frac{s}{\rho}\tilde g_+(\rho e_k,\tau_1)\\
={}& \left(1-\frac s\rho\right)g_+(0,\tau_1)+\frac{s}{\rho}\tilde g_-(\rho e_k,\tau_1)\,.
\eann
Next, observe that, since the function $s\mapsto \tilde g_-(se_k,\tau_1)$ is convex and satisfies $\frac{d}{ds} \tilde g_-(s e_k, \tau_1)\big|_{s=0}=D_k\tilde g_-(0,\tau_1)\ge 0$, it  must be monotone non-decreasing for $s\in [0,\rho]$. Thus,
\bann
\tilde g_+(se_k,\tau_1)\ge{}&g_+(0,\tau_1)-\frac{s}{\rho}\left(g_+(0,\tau_1)-\tilde g_-(0,\tau_1)\right)\\
={}&g_+(0,\tau_1)-\frac{s}{\rho}\left(g_+(0,\tau_1)-g_-(0,\tau_1)\right)\,.
\eann
This yields \eqref{eq:horizontal displacement first case} since, for $s\in [0, \frac\rho2]$, the right hand side is minimized at $s=\frac\rho 2$ (and $\rho\ge \frac{L}{2C}$). 

Since $\frac{d}{ds} \tilde g_+(s e_k, \tau_1)\big|_{s=0}=0$, 
essentially the same argument yields the estimate
\begin{equation}\label{eq:horizontal displacement second case}
\tilde g_-(se_k, \tau_1) \leq b\;\; \text{for}\;\; s \in [-\tfrac{L}{4C}, \tfrac{L}{4C}]\,.
\end{equation}

Finally, to bound $\tilde g_+(s e_k, \tau_1)$ for $s<0$, we argue as follows.
Since $s\mapsto \tilde g_-(s e_k, \tau_1)$ is convex,
\bann
\tilde g_-(0, \tau_1)-\tilde g_-(-\sigma e_k, \tau_1)\le{}& \frac{\sigma}{\rho} (\tilde g_-(\rho e_k, \tau_1)-\tilde g_-(0,\tau_1))\\
={}& \frac{\sigma}{\rho} (\tilde g_+(\rho e_k, \tau_1) - g_-(0, \tau_1))\\
\le{}& \frac{\sigma}{\rho} (\tilde g_+(0, \tau_1) - g_-(0, \tau_1))\\
={}& \frac{\s}{\rho}g(0, \tau_1),
\eann
and, since $s\mapsto \tilde g_+(s e_k, \tau_1)$ is concave, we have for $s\in[-\sigma,0]$ 
\[
\begin{split}
\tilde g_+(s e_k, \tau_1)&\ge \left(1+\frac s\sigma\right)\tilde g_+(0, \tau_1)-\frac{s}{\sigma}\tilde g_+(-\sigma e_k, \tau_1)\\
&= \left(1+\frac s\sigma\right)g_+(0, \tau_1)-\frac{s}{\sigma}\tilde g_-(-\sigma e_k, \tau_1)\\
&\ge \left(1+\frac s\sigma\right)g_+(0, \tau_1)+\frac{s}{\rho}g(0, \tau_1)-\frac{s}{\sigma}\tilde g_-(0,\tau _1)\\
&=b+\left(\frac 12+\frac s\sigma+\frac s\rho\right) g(0, \tau_1)\,.
\end{split}
\]
Since $\min\{\rho,\sigma\}\ge \frac{L}{2C}$, we conclude that
\begin{equation}\label{eq:horizontal displacement third case}
\tilde g_+(s e_k, \tau_1)\ge b\;\;\text{for}\;\; s\in[-\tfrac{L}{8C},0]\,.
\end{equation}

To summarize, \eqref{eq:horizontal displacement first case}--\eqref{eq:horizontal displacement third case} show that, for every integer $2 \leq k \leq n$,
\[\tilde g_-(se_k, \tau_1) \leq b \leq \tilde g_+(se_k, \tau_1)\;\;\text{for}\;\; |s| \leq \tfrac{L}{8C}.\]
Finally, we observe that a similar argument yields
\[g_-(se_1, \tau_1) \leq b \leq g_+(se_1, \tau_1)\;\;\text{for}\;\; |s| \leq \tfrac{L}{8C}.\]
In fact, this case is a little simpler, since we arranged that $\mathbb R_-e_1 \subset \Omega_{t}$ for $t \leq 0$, which ensures that $g_+(se_1, \tau_1)$ and $g_-(se_1,\tau_1)$ are nonincreasing and nondecreasing, respectively. 

Since $\tilde g_{\pm}(se_1, \tau_1) = g_\pm (se_1, \tau_1)$, we finally conclude that
\[\tilde g_-(x, \tau_1) \leq b \leq \tilde g_+ (x, \tau_1) \;\;\text{for}\;\; x \in T,\]
where $T$ is the convex hull of $\{\pm \tfrac{ L}{8C} e_k\}_{1 \leq k \leq n}$. Choosing $\mu = \mu(n)$ such that $B^n_{2\mu L/\pi C}(0) \subset T$ and inserting the definition of $L$ we arrive at
\[\tilde g_-(x, \tau_1) \leq b \leq \tilde g_+ (x, \tau_1)\;\;\text{for}\;\; x \in B^n_{\mu (\tau_1 - \tau_0)/C}(0),\]
which is to say 
\[g_-(x,\tau_1) \leq b + v \cdot x \leq g_+(x, \tau_1)\;\;\text{for}\;\;  x \in B^n_{\mu (\tau_1 - \tau_0)/C}(0).\qedhere\]
\end{proof}

Using Lemma \ref{lem:width} we now establish a result of a similar nature under different hypotheses. This is a key ingredient in the base step of the induction argument used to prove Proposition \ref{prop:slab}.

\begin{lemma}
\label{lem:width_2}
There are constants $\beta = \beta(n)> 0$ and $\Lambda = \Lambda(n) < \infty$ with the following property. If $M_{-1} \cap B_R(0) \subset \{|x_{n+1}| \leq \beta\}$ and $R \geq \Lambda \beta^{-1}$, then there exist $x_0 \in B^n_R(0)$, $b\in \R$, and $v\in \R^n$ with $\vert v\vert\le 1/2$ such that
\[g_-(x,1) \leq b + v\cdot (x-x_0) \leq g_+(x,1)\;\;\text{for every}\;\; x \in B^n_{\mu /4 \beta}(x_0)\,.\]
\end{lemma}
\begin{proof}
Given $\beta$ sufficiently small and $R$ sufficiently large, we will find $x_0\in B_{R}(0)$ such that $|Dg_\pm(x_0, \tau)|\le 1/2$ for all $ \tau \in [1/2, 1]$ and apply Lemma \ref{lem:width} with $C= 2\beta$.

Arguing similarly as in Lemma \ref{lem:width}, using the area enclosed by $M_t\cap \Sigma$ for $\Sigma  \doteqdot \spa\{e_k, e_{n+1}\}$ and any $2\le k\le n$, we obtain that for $\beta \leq 1$ and $R\ge 2\beta^{-1}$, after relabelling axes if necessary,  $U_{1/2}$ contains the convex hull of $\{0,-\beta^{-1}e_1, \beta^{-1}e_2, \dots, \beta^{-1} e_n \}$. The largest ball inscribed in this hull is of the form $B^n_{c/\beta}(x_0)$, where $x_0 \in B^n_{1/\beta}(0)$ and $c = c(n)$, so we have 
\[
B^n_{c/\beta}(x_0) \subset U_{1/2} \subset U_\tau\;\;\text{for all}\;\; \tau \geq \tfrac{1}{2}\,.
\]
Taking $ R \geq 2(1+c) \beta^{-1}$ ensures that $|g_{\pm}(\cdot,1)|\leq \beta$ in $B^n_{c /\beta}(x_0) $, and since $g_+(\cdot, \tau)$ is concave we may estimate, for each $ 1 \leq k \leq n$ and $\tfrac{1}{2} \leq \tau \leq 1$,
\[D_k g_+(x_0, \tau) \leq c^{-1} \beta (g_+(x_0,\tau) - g_+(x_0 - c\beta^{-1} e_k, \tau)) \leq 2 c^{-1} \beta^2,\]
and similarly $|D g_{\pm} (x_0, \tau) |\le 2 c^{-1}\beta^2$.
Therefore, if $\beta$ is sufficiently small we have 
\[|D g_{\pm} (x_0, \tau) | \leq 1/2\;\;\text{for}\;\; \tau \in [\tfrac{1}{2}, 1].\]
Now apply Lemma \ref{lem:width}.   
\end{proof}

\subsection{Vertical displacement bounds}  The next lemma gives conditions under which $g_+(0,\tau)$ grows at most linearly in $\tau$ over a given time interval.

\begin{lemma}
\label{lem:height}
Suppose that 
\begin{equation}
\label{eq:height_ass}
g_-(x, \tau_0) \leq b + v \cdot x \leq g_+(x,\tau_0)\;\;\text{for all}\;\; x \in B_R^n(0)
\end{equation}
for some $b \in \mathbb{R}$ and $v\in \R^n$ with $|v| \leq 1$. If $\tau_1 > \tau_0$ is such that $\delta \doteqdot g_+(0, \tau_1) - g_+(0, \tau_0)$ satisfies 
\[g_+(0,\tau_0)  - b \leq \delta \leq \frac{R}{100},\]
then we have the estimate 
\[\tau_1 - \tau_0 \geq \xi \delta R,\]
where $\xi = \xi(n)$. 
\end{lemma}
\begin{proof}
Denote by $\Sigma$ the hyperplane $\{b + v \cdot x : x \in \mathbb{R}^n\}$. Since $|v| \leq 1$ we can choose a unit normal $\nu_\Sigma$ to this plane satisfying $\nu_\Sigma \cdot e_{n+1} \geq \tfrac{1}{\sqrt{2}}$. Observe that \eqref{eq:height_ass} implies 
\[\Sigma\cap B_R(b e_{n+1}) \subset \Omega_{t_0},\]
where $t_0 \doteqdot -\tau_0$ (we also write $t_1 \doteqdot -\tau_1$). Since $\Omega_t$ is convex we conclude that $\overline{\Omega}_{t_1}$ contains the convex hull of $p_1 \doteqdot g_+(0, \tau_1) e_{n+1}$ and $\Sigma \cap B_R(b e_{n+1})$, which we denote by $K$. 

To prove the claim, we enclose $p_0 \doteqdot g_+(0,\tau_0)e_{n+1}$ with a suitable rescaled timeslice of the ancient pancake solution, whose forward evolution then acts as a barrier. First, we need to identify a suitable cylinder region in $K$. To this end, we define, for each $p \in \mathbb{R}^{n+1}$, the projections
\[\pi^\perp(p) \doteqdot ((p - p_0)\cdot \nu_\Sigma )\, \nu_\Sigma\;\;\text{and}\;\;  \pi^\top(p) \doteqdot (p-p_0) - \pi^\perp(p)\,.\]

\begin{claim}\label{claim:cylinder} Set $\bar R = 10^{-2} R$ and $\bar \delta = 10^{-2} \delta$. The truncated cylinder 
\[Z\doteqdot \big \{p \in \mathbb{R}^{n+1} : |\pi^\perp(p) | \leq \bar \delta ,\; \;  |\pi^\top(p)| \leq \bar R \big \}\]
is contained in $K$.
\end{claim}
\begin{proof}[Proof of Claim \ref{claim:cylinder}]
Let 
\[D_\rho \doteqdot (\Sigma - be_{n+1}) \cap B_{\rho}(0)\]
and define
\[R_0 \doteqdot \sup\{ \rho \geq 0  : D_\rho + p_0 \subset K\}.\]
That is, $D_{R_0} + p_0$ is the largest round disc parallel to $\Sigma$ which passes through $p_0$ and sits inside $K$.  Using similar triangles we find
\[\frac{g_+(0,\tau_1) - g_+(0,\tau_0)}{R_0} = \frac{g_+(0,\tau_1) - b}{R} ,\]
and combining this with our assumption $\delta \geq g_+(0,\tau_0) - b$ shows $R_0 \geq R/2$. Similarly, the quantity
\[R_1 \doteqdot \sup\{\rho \geq 0 : D_\rho + p_0 + \tfrac{\delta}{10} e_{n+1}\subset K\} \]
satisfies 
\[\frac{9}{10} \frac{\delta}{R_1} =\frac{g_+(0,\tau_1) - g_+(0,\tau_0)}{R_0}= \frac{\delta}{R_0},\]
hence $R_1 \geq \tfrac{9}{20} R$. It follows that 
\[ D_{9 R/20} + a e_{n+1} \subset K \subset \Omega_{t_1}\]
for all $a \in [b, g_+(0,\tau_0) + \tfrac{\delta}{10}]$, and in particular for all $a$ satisfying
\[|a - g_+(0, \tau_0)| \leq \tfrac{\delta}{10},\]
hence the family of vertically stacked discs
\[Z' \doteqdot \bigcup_{|a - g_+(0, \tau_0)| \leq \delta/10} D_{9 R/20} + a e_{n+1} \]
is contained in $K$.  Using $\nu_\Sigma \cdot e_{n+1} \geq \tfrac{1}{\sqrt{2}}$ we find 
\[\Big \{p \in \mathbb{R}^{n+1} : |\pi^\perp(p) | \leq  \frac{\delta}{10 \sqrt{2}} ,\; \;  |\pi^\top(p)| \leq \frac{9}{20} \Big( 1 - \frac{1}{\sqrt{2}}\Big) R\Big \} \subset Z',\]
hence $Z \subset Z' \subset K$. 
\end{proof}

Now let $\{P_t\}_{t \in (-\infty, 0)}$ be the rotationally symmetric pancake \cite{BLT1} which contracts to $p_0$ as $t \to 0$ and satisfies
\[ \cup_{t < 0} \, P_t = \{| \pi^\perp(p) | \leq \bar \delta \}\setminus \{0\} ,\]
and let $-\bar \tau$ be the earliest time such that $P_{-\bar \tau}$ is a subset of $Z$. Then $P_{-\bar \tau}$ has extrinsic diameter
\[2\bar R \geq 2 \delta \geq 200 \bar \delta. \] 
That is, the diameter of $P_{-\bar \tau}$ is at least 200 times larger than the width of the slab swept out by $P_t$ as $t \to - \infty$. In this regime, by the estimates in \cite{BLT1}, there is a universal constant $C$ such that the diameter of $P_{-\bar \tau}$ does not exceed $C\delta^{-1} \bar \tau$, so we find 
\[C\delta^{-1} \bar \tau \geq 2 \bar R.\]
On the other hand, $p_0 \in M_{-\tau_0}$, so the avoidance principle implies $\tau_1 - \tau_0 \geq \bar \tau$, and hence
\[\tau_1 - \tau_0 \geq  2\delta C^{-1}  \bar R \geq \frac{2}{100}  C^{-1} \delta R,\]
which is the desired estimate. 
\end{proof}

\subsection{Wang's iteration} We now have all the auxiliary results needed to carry out X.-J. Wang's iteration. This will complete the proof of the slab estimate. 

\begin{proof}[Proof of Proposition \ref{prop:slab}] 
Suppose that $M_{-1} \cap B_R(0) \subset \{|x_{n+1}|\leq \beta\}$, where $\beta \in (0,1)$ is a small constant to be determined, and $\beta$ and $R$ are such that Lemma~ \ref{lem:width_2} applies. Then we can find a point $x_0 \in B^n_R(0)$ such that 
\begin{equation}
\label{eq:init_disc}
g_-(x,1) \leq b_0 + v_0 \cdot (x-x_0) \leq g_+(x,1) \;\;\text{for}\;\; x \in B^n_{L} (x_0),
\end{equation}
where $b_0 \in \mathbb{R}$, $|v_0| \leq 1/2$, and $L\doteqdot \tfrac{\mu}{4\beta}$ with $\mu=\mu(n)$ as in Lemma \ref{lem:width}. Performing a translation if necessary we may assume $x_0 = 0$. 

Set $h_k = 2^k$ and $g_k(\cdot) = g(\cdot, h_k)$ for each $k\ge 0$. We will demonstrate that, for $\beta$ sufficiently small, the inequality
\begin{equation}
\label{eq:iteration}
g_{m}(0) \le g_{m-1}(0) + 2^{-\gamma m}
\end{equation}
holds for every $m \geq 0$ for some $\gamma>0$. So let $k_0$ be a large integer and $\gamma$ a small positive number, which we will refine in the course of the proof. 

\begin{claim}
\label{claim:base}
If $\beta$ is sufficiently small (depending only on $n$, $k_0$ and $\gamma$), then \eqref{eq:iteration} holds for $0 \leq m \leq k_0$, and
\begin{equation}
\label{eq:g_k0bd}
g_{k_0}(0) \leq 3.
\end{equation}
\end{claim}
\begin{proof}[Proof of Claim \ref{claim:base}]
Set $\delta = 2^{- \gamma k_0-1}$. Without loss of generality, we can find some $h \geq 1$ such that
\[g_+(0,h) - g_+(0, 1) = \delta.\]
If $\beta$ is small enough that 
\[g_+(0,1) - b_0 \leq 2 \beta \leq \delta \leq \frac{L}{100},\]
then by \eqref{eq:init_disc} and Lemma \ref{lem:height} we have
\[h - 1 \geq \xi \delta L =  2^{-\gamma k_0 - 3} \mu \xi  \beta^{-1}.\]
Therefore, as long as $\beta$ is sufficiently small, we find that $h \geq 2^{k_0}$, and hence 
\[g_+(0, 2^{k_0}) \leq g_+(0, h) \leq g_+(0,1) + \delta = g_+(0,1) + 2^{- \gamma k_0-1}.\]
An almost identical argument leads to the estimate 
\[- g_-(0, 2^{k_0}) \leq -g_-(0,1) + 2^{- \gamma k_0-1},\]
so we have $g_{k_0}(0) \leq g(0, 1) + 2^{- \gamma k_0}$, and since we are assuming $\beta \leq 1$ this ensures \eqref{eq:g_k0bd} holds. In addition, for $1 \leq m \leq k_0$ we conclude 
\[g_m(0) \leq g_{k_0}(0) \leq g(0,1) + 2^{-\gamma k_0} \leq g_{m-1}(0) + 2^{- \gamma m},\]
which is \eqref{eq:iteration}.
\end{proof}

Next, assume \eqref{eq:iteration} holds for all $k_0 \leq m \leq k$. We will show that, for $k_0$ sufficiently large and $\beta$ sufficiently small, this implies \eqref{eq:iteration} holds also for $m=k+1$. By induction, this will establish \eqref{eq:iteration} for all $m \geq 0$. We can proceed in much the same way as in \cite{Wa11}, provided we first prove:

\begin{claim}
\label{claim:width&speed}
If $k_0$ is sufficiently large (depending only on $n$) and $\beta$ is sufficiently small (depending only on $n$, $k_0$ and $\gamma$), and \eqref{eq:iteration} holds for all $k_0 \leq m \leq k$, then the following properties hold for some positive constants $c_0$ and $C_0$ which depend only on $n$.
\begin{enumerate}[(i)]
\item \label{cond:width} 
There exist $b_k \in \mathbb{R}$ and $v_k \in \mathbb{R}^n$ with $|v_k| \leq 1$ such that
\begin{equation}
\label{eq:g_affine}
g_-(x, h_k) \leq b_k + v_k \cdot x \leq g_+(x, h_k)\;\;\text{for}\;\; x \in B^n_{c_0 h_k}(0)\,. 
\end{equation}
\item \label{cond:speed} There is a universal $\eta>0$ 
such that
\[
|\partial_h g(x, h)| \leq C_0 h_k^{-\eta}\;\;\text{for each}\;\; (x, h) \in B^n_{c_0 h_k /2} \times [h_k, h_{k+1}]\,.
\]
\end{enumerate}
\end{claim}

\begin{proof}[Proof of Claim \ref{claim:width&speed}]
In what follows $\Lambda_i$ always denotes a large constant depending only on $n$. Let $\beta$ be such that Claim \ref{claim:base} holds. 

To prove \eqref{cond:width}, we first use the inductive hypothesis and \eqref{eq:g_k0bd} to bound
\[g_{k}(0) \leq g_{k_0}(0) + \sum_{m = k_0}^{k-1}(g_{m+1}(0) - g_m(0)) \leq 3 + \sum_{m = 0}^\infty 2^{-\gamma (m+1) }\leq \Lambda_1, \]
and thus deduce $|g_{\pm}(0,2^k)| \leq 2\Lambda_1$. In light of \eqref{eq:init_disc} we also know
\[\pm \tilde g_{\pm}(x,h) \doteqdot \pm (g_{\pm} (x, h) - b_0 - v_0\cdot x)\]
is concave and nonnegative in $B^n_L(0)$ for $h \geq 1$, so recalling that $L = \mu/4\beta$ and $|b_0|\leq 2\beta\leq 2$, we obtain an estimate of the form
\[|D \tilde g_{\pm}(0,h)| \leq \Lambda_2 \beta\]
valid for $1 \leq h \leq h_k$. Assuming $\beta$ is sufficiently small, we can arrange 
\[|D \tilde g_{\pm}(0, h)| \leq 1/2\]
for $1 \leq h \leq h_k$, which implies 
\[|D g_{\pm} (0, h)| \leq 1/2 + |v_0| \leq 1\]
for $h \in [h_k, h_{k+1}]$. Appealing to Lemma \ref{lem:width} we find there is a $c_0 = c_0(n)$ such that 
\[g_-(x, 2^k) \leq b_k + v_k \cdot x \leq g_+(x, 2^k) \;\; \text{for}\;\; x \in B^n_{c_0 h_k}(0),\]
where $b_k \in \mathbb{R}$ and $|v_k| \leq 1$. 

Next we prove \eqref{cond:speed} using the Ecker--Huisken curvature estimate for graphs. Fix $\delta = 2 \Lambda_1 + 4c_0^{-1} \xi^{-1} $ and let $h$ be such that $g_+(0, h) - g_+(0, 2^k) = \delta$. Then we certainly have $g_+(0,2^k) - b_k \leq 2\Lambda_1  \leq \delta$, and if $k_0$ is large enough to guarantee $\delta \leq \tfrac{c_0}{100} 2^{k_0}$, then we have $\delta \leq \tfrac{c_0}{100} h_k$ for all $k \geq k_0$, so Lemma \ref{lem:height} implies 
\[h - 2^k \geq c_0 \xi \delta h_k \geq 4 h_k.\]
In particular, $h \geq 2^{k+2}$, so we have 
\[g_+(0, 2^{k+2}) \leq g_+(0, h) \leq g_+(0, 2^k) + \delta \leq \Lambda_3.\] 
The same argument shows $-g_-(0,2^{k+2}) \leq \Lambda_3$, so since \eqref{cond:width} implies the functions $x \mapsto \pm (g_{\pm}(x, h) - b_k - v_k \cdot x)$
are concave and nonnegative for $x \in B_{c_0 h_k}^n(0)$ and $h \geq h_k$, we obtain an estimate of the form
\[|D g_{\pm}(x,h)| \leq \Lambda_4 h_k^{-1} + 1 \leq \Lambda_4 + 1\]
for all $(x, h) \in B^n_{3c_0 h_k/4}(0) \times [h_k, h_{k+2}]$.


The estimate of Ecker--Huisken we make use of is as follows (see \cite[Theorem 3.1]{EckerHuisken91}): if $M_t$ is a graph over $B^n_{3 \rho/4}(0) \times [-T, 0]$ with gradient bounded above by $\Lambda$, then 
\[\sup_{B^n_{\rho/2}(0)} H(\cdot ,0) \leq C(n, \Lambda) (\rho^{-1} +  T^{-1/2}).\]
Given any $(x, h) \in B^n_{c_0 h_k/2}(0) \times [h_k, h_{k+1}]$ we may apply this estimate with $\rho = 3 c_0 h_k/4$ and $T = h_{k+2} - h \geq 2 h_k$ to obtain
\[|\partial_h g(x, h)| \leq C(n,c_0, \Lambda_4) (h_k^{-1} + h_k^{-1/2}) \leq C_0 h_k^{-1/2},\]
where $C_0$ depends only on $n$. So \eqref{cond:speed} holds with $\eta=\frac{1}{2}$.
\end{proof}

The next step is to show that properties \eqref{cond:width} and \eqref{cond:speed} imply that \eqref{eq:iteration} holds for $m= k+1$. This is achieved by adapting arguments in \cite{Wa11}. 

\begin{claim}
\label{claim:inductivestep}
If $k_0$ is sufficiently large (depending only on $n$) and $\beta$ is sufficiently small (depending only on $n$, $k_0$ and $\gamma$), and \eqref{eq:iteration} holds for all $k_0 \leq m \leq k$, then \eqref{eq:iteration} holds also for $m = k+1$. 
\end{claim}

\begin{proof}[Proof of Claim \ref{claim:inductivestep}]
We write $\Lambda$ and $\Lambda_i$ for large constants depending on $n$, and allow the value of $\Lambda$ to increase from line to line. Let $k_0$ and $\beta$ be such that Claims \ref{claim:base} and \ref{claim:width&speed} hold. 

Recall that in the proof of Claim~\ref{claim:width&speed} we used the inductive hypothesis to conclude $|g_{\pm}(0,2^k)| \leq \Lambda$, and Lemma \ref{lem:height} to conclude $|g_{\pm}(0, 2^{k+2})| \leq \Lambda$. In particular, $g_{k+1}(0)\leq \Lambda$. 

Property \eqref{cond:width} of Claim \ref{claim:width&speed} implies $B^n_{c_0 h_k}(0) \subset U_{h_k}$, so we have $B^n_{c_0 h_k}(0) \subset U_h$ for all $h \geq h_k$, hence 
\[Q_k \doteqdot B^n_{c_0 h_k/2}(0) \times [h_k, h_{k+1}]\]
is contained in $U$. Since $g(\cdot ,h )$ is concave and nonnegative for each $h$, and $g(0, h) \leq \Lambda$ for all $h \leq h_{k+1}$, we have the gradient bound
\begin{equation}
\label{eq:deriv_est_space}
|D_i g(x, h)| \leq \Lambda h_{k}^{-1}
\end{equation}
for all $1 \leq i \leq n$ and $(x,h) \in Q_k$. Furthermore, property \eqref{cond:speed} of Claim \ref{claim:width&speed} provides some $\eta\in(0,1)$ such that
\begin{equation}
\label{eq:deriv_est_time}
|\partial_h g(x,h)| \leq C_0 h_k^{-\eta}
\end{equation} 
in $Q_k$.

Let $\varepsilon \in (0,1)$ and $\delta\in(0,1)$ be constants to be chosen later. Let $B_k \doteqdot B^n_{c_0 h_k/2}(0)$, and define
\[\chi \doteqdot \{(x, h) \in Q_k : -\Delta g(x,h) \geq h_k^{-2+\varepsilon}\},\]
where $\Delta$ is the Laplacian in $(x_1, \dots, x_n)$. Using $\Delta g \leq 0$ and \eqref{eq:deriv_est_space} we estimate, for $h_k \leq h \leq h_{k+1}$,
\begin{align*}
|\{x \in B_k : (x, h) \in \chi\}| \, h_k^{-2+\varepsilon} &\leq -\int_{B_{k}} \Delta g  = \int_{\partial B_k} D g \cdot  \tfrac{x}{|x|}  \leq \Lambda h_k^{n-1} h_k^{-1},
\end{align*}
and thus obtain
\[|\{x \in B_k : (x, h) \in \chi\}| \leq \Lambda h_k^{n - \varepsilon}.\]
Integrating now over $h\in (h_k, h_{k+1})$ gives
\[|\chi| \leq \Lambda_5 h_k^{n+1-\varepsilon}.\]

For each $y \in B_k$, let us write 
\[\chi_{y} \doteqdot \{h \in [h_k, h_{k+1}] : (y, h) \in \chi\}.\]
If $\tilde B$ is the set of $y \in B_k$ such that $|\chi_{y}| \geq \Lambda_5 h_k^{\delta}$, by Fubini's theorem,
\[\Lambda_5 |\tilde B| h_k^{\delta} \le \int_{B_k \setminus \tilde B} |\chi_{y}| + \int_{\tilde B} |\chi_{y}| = \int_{B_k} |\chi_{y}| = |\chi| \leq \Lambda_5 h_k^{n+1-\varepsilon},\]
so $|\tilde B| \le h_k^{n+1-\delta -\varepsilon}$. For every $y \in B_k \setminus \tilde B$ we have 
\begin{align*}
g_{k+1}(y) - g_k(y) &= \int_{\chi_{y}} \partial_h g(y,s)\,ds + \int_{[h_{k+1}, h_k]\setminus \chi_{y}} \partial_h g(y,s)\,ds\\
&\leq |\chi_{y}| \sup_{Q_k} |\partial_h g| + \int_{[h_{k+1}, h_k] \setminus \chi_{y}} \partial_h g(y,s)\,ds\\
& \leq \Lambda h_k^{\delta -\eta} + \int_{[h_{k+1}, h_k]\setminus \chi_{y}} \partial_h g(y,s)\,ds,
\end{align*}
where we have made use of \eqref{eq:deriv_est_time}. To estimate the second integral on the right we use the fact that $M_t$ is convex and evolves by mean curvature flow to obtain
\begin{align*}
\partial_h g &= \partial_h g_+ - \partial_h g_- \\
 &= -\left(\delta^{ij} - \frac{D_i g_+ D_j g_+}{1+|Dg_+|^2}\right) D_i D_j g_+ + \left(\delta^{ij} - \frac{D_i g_- D_j g_-}{1+|Dg_-|^2}\right) D_i D_j g_-\\
 &\leq - \Delta g_+ + \Delta g_-\\
 &= - \Delta g,
\end{align*}
and then insert the definition of $\chi$ to arrive at
\[\int_{[h_{k+1}, h_k]\setminus \chi_{y}} \partial_h g(y,s)\, ds \leq \Lambda h_k^{-1 + \varepsilon}.\]
With this we have 
\[g_{k+1}(y) - g_k(y)  \leq \Lambda h_k^{\delta -\eta} + \Lambda h_k^{-1 + \varepsilon}.\]
By choosing $\delta = \eta /2$ and $\varepsilon = 1 - \eta /4$ we ensure 
\[g_{k+1}(y) - g_k(y)  \leq \Lambda h_k^{-\eta/2} + \Lambda h_k^{- \eta/4} \leq \Lambda h_k^{-\eta/4}, \]
and 
\[|\tilde B| \leq h_k^{n -\eta/4}.\]

The last inequality implies we can find a point $y_0 \in B_k \setminus \tilde B$ such that $|y_0| \leq \Lambda h_k^{1-\frac{\eta}{4n}}$. Integrating \eqref{eq:deriv_est_space} we find 
\[|g_{k}(y_0) - g_k(0)| \leq \Lambda h_k^{-\frac{\eta}{4n}}\]
and 
\[|g_{k+1}(y_0) - g_{k+1}(0)| \leq \Lambda h_k^{-\frac{\eta}{4n}},\]
so we finally arrive at 
\[|g_{k+1}(0) - g_k(0)|\leq \Lambda h_k^{-\frac{\eta}{4n}} + |g_{k+1}(y_0) - g_k(y_0)| \leq \Lambda h_k^{-\frac{\eta}{4n}}.\]
If $k_0$ is sufficiently large this implies 
\[g_{k+1}(0) \leq g_k(0) + 2^{-\frac{\eta}{8n} (k+1)}.\]
Taking $\gamma=\eta/8n$, we obtain \eqref{eq:iteration} with $m= k+1$, as required.
\end{proof} 


In summary, if we set $\gamma=\frac{\eta}{8n}$, then (since $\eta$ is universal
) we can fix $k_0$ sufficiently large (depending only on $n$) and $\beta$ sufficiently small (depending only on $n$) such that Claims \ref{claim:base}, \ref{claim:width&speed} and \ref{claim:inductivestep} hold with this choice of $\gamma$. 
Therefore, by induction, \eqref{eq:iteration} holds for all $m \geq 0$. From this we conclude
\[|\Omega_t \cap \mathbb{R} e_{n+1}| \leq g(0,1) + \sum_{m =1}^\infty (g_m(0)-g_{m-1}(0)) \leq  2\beta + \frac{1}{2^{\gamma}-1}\]
for all $t \leq 0$. Consequently, $\dist(0, M_{t})$ is bounded independently of $t$, so there is a sequence of spacetime points $(x_i,t_i)$ such that $x_i$ is bounded, $t_i \to - \infty$, and 
\[H(x_i,t_i) \to 0.\]
Appealing to Proposition \ref{prop:interior_est} we find that the sequence $\{M_{t+t_i}\}_{t \in (-\infty, -t_i]}$ subconverges in $C^\infty_{\loc}(\mathbb{R}^{n+1} \times \mathbb{R})$ to a convex limit which contains a point of zero mean curvature. The strong maximum principle then implies that it is stationary, and thus consists of two parallel hyperplanes bounding a slab region. In particular, by monotonicity of the enclosed regions $\Omega_t$, $M_t$ lies in a fixed slab for all $t \leq 0$.
\end{proof}

\bibliographystyle{plain}
\bibliography{bibliography}

\end{document}